\documentclass{amsart}
\usepackage{amsthm,amsmath,amsfonts,amssymb}
\usepackage[numbers,sort&compress]{natbib}
\usepackage[colorlinks,citecolor=blue,urlcolor=blue]{hyperref}
\usepackage{graphicx}
\usepackage[british]{babel}
\usepackage[usenames,dvipsnames]{xcolor}
\usepackage{anyfontsize}
\usepackage{comment}
\usepackage{dsfont}
\usepackage{mathrsfs}
\usepackage{mathabx}
\usepackage{xfrac}
\usepackage{mathtools}
\usepackage{wrapfig}

\usepackage[expansion=true, protrusion=true]{microtype}
\usepackage[all]{xy}
\usepackage{float}
\usepackage{subcaption}

\usepackage[T1]{fontenc}
\usepackage{geometry}

\usepackage[inline]{enumitem}

\usepackage{tabularx}
\usepackage{longtable}

\usepackage{afterpage}

\definecolor{brick}{HTML}{FF0800}

\theoremstyle{plain}
\newtheorem{theorem}{Theorem}
\newtheorem{lemma}[theorem]{Lemma}
  
  \newtheorem{corollary}[theorem]{Corollary}
  
  \newtheorem{fact}[theorem]{Fact}

\theoremstyle{definition}
  
  \newtheorem{definition}[theorem]{Definition}
  
  \newtheorem{example}[theorem]{Example}

  \newtheorem{condition}[theorem]{Condition}
  \newtheorem{ass}{Assumption}

\theoremstyle{remark}
  \newtheorem{remark}[theorem]{Remark}

  \newtheorem{application}[theorem]{Application}

\DeclareMathOperator{\proj}{proj}
\DeclareMathOperator{\diam}{diam}

  \DeclareMathOperator\Leb{Leb}

  \DeclareMathOperator{\dimH}{dim_H}
  \DeclareMathOperator{\dimP}{dim_P}
  \DeclareMathOperator{\dimB}{dim_B}
  \DeclareMathOperator{\dimS}{dim_S}

\def \N {\mathbb N}

\def \R {\mathbb R}
\def \ind{1\!\!1}

\newcommand*{\e}[1]{\text{e}^{#1}}

\title[Random self-similar IFSs]{Dimension Properties of a Family of Random Self-Similar Iterated Function Systems}

\author{Vilma Orgoványi}
\address{Department of Stochastics, Institute of Mathematics, Budapest University of Technology and Economics, Műegyetem rkp.~3, H-1111 Budapest, Hungary}
\curraddr{HUN-REN Alfréd Rényi Institute of Mathematics, Reáltanoda utca 13--15, H-1053 Budapest, Hungary}
\email{orgovanyi.vilma@gmail.com}

\author{Károly Simon}
\address{Department of Stochastics, Institute of Mathematics, Budapest University of Technology and Economics, Műegyetem rkp.~3, H-1111 Budapest, Hungary}
\address{HUN-REN--BME Stochastics Research Group, Műegyetem rkp.~3, H-1111 Budapest, Hungary}
\email{karoly.simon51@gmail.com}

\begin{document}
\begin{abstract}
In this note we present a dimension formula for certain randomized self-similar sets. The self-similar sets we are considering are $d$-dimensional generalizations of rational projections of inhomogeneous sponges: they have overlaps, but the overlapping structure is in some sense well-organized. The randomization happens according to a labelled Galton--Watson tree corresponding to the cylinders of the IFS, for example as in the case of the Mandelbrot percolation set. The dimension formula is given in terms of a pressure function corresponding to the expectation matrices, which describes the overlapping structure and the randomness of the system.
\end{abstract}
\maketitle
\section{Introduction}
In this note we consider the Hausdorff, box and packing dimensions of some overlapping statistically self-similar sets. A special case of the random model we are considering is the Mandelbrot percolation process running on cylinders of certain self-similar iterated function systems (IFSs). The most natural examples of the one-dimensional IFSs in which we are interested are the rational projections of $2$-dimensional sponges. The projected IFS is conjugated by a similarity (see Section \ref{conjugation}) to an integer self-similar IFS. 
Our running examples are the following IFSs:
\begin{example}\label{ex:1d}
    The following IFS is conjugated by a similarity to the $45$-degree projection of the right-angled Sierpi\'nski carpet, see Figure \ref{fig:IFS}. 
    \begin{equation}
     \mathcal{S}=\left\{ S_i(x)=\frac{1}{2}x+i \right\}_{i=0}^2 
    \end{equation}
\end{example}
\begin{example}\label{ex:2d}
        The following is one of the simplest 2-dimensional examples for which the results of this note can be applied. See Figure \ref{fig:IFS} depicting the first level cylinders. The points for the translations are denoted by red.
    \begin{equation}
        \mathcal{S} =\left\{\frac{\underline{x}}{2}+\underline{t}_i\right\}_{i=1}^{9},
    \end{equation}
    where $\underline{t}_i$ runs through the set $\{0, 1, 2\}^2$.
\end{example}

   \begin{figure}
        \centering
        \includegraphics[width=0.315\linewidth]{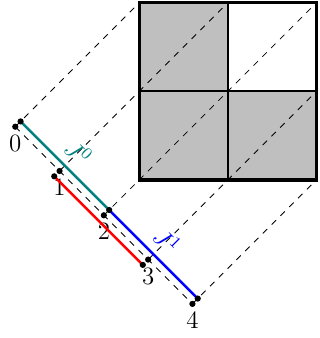}\hspace{1cm}\includegraphics[width=0.352\linewidth]{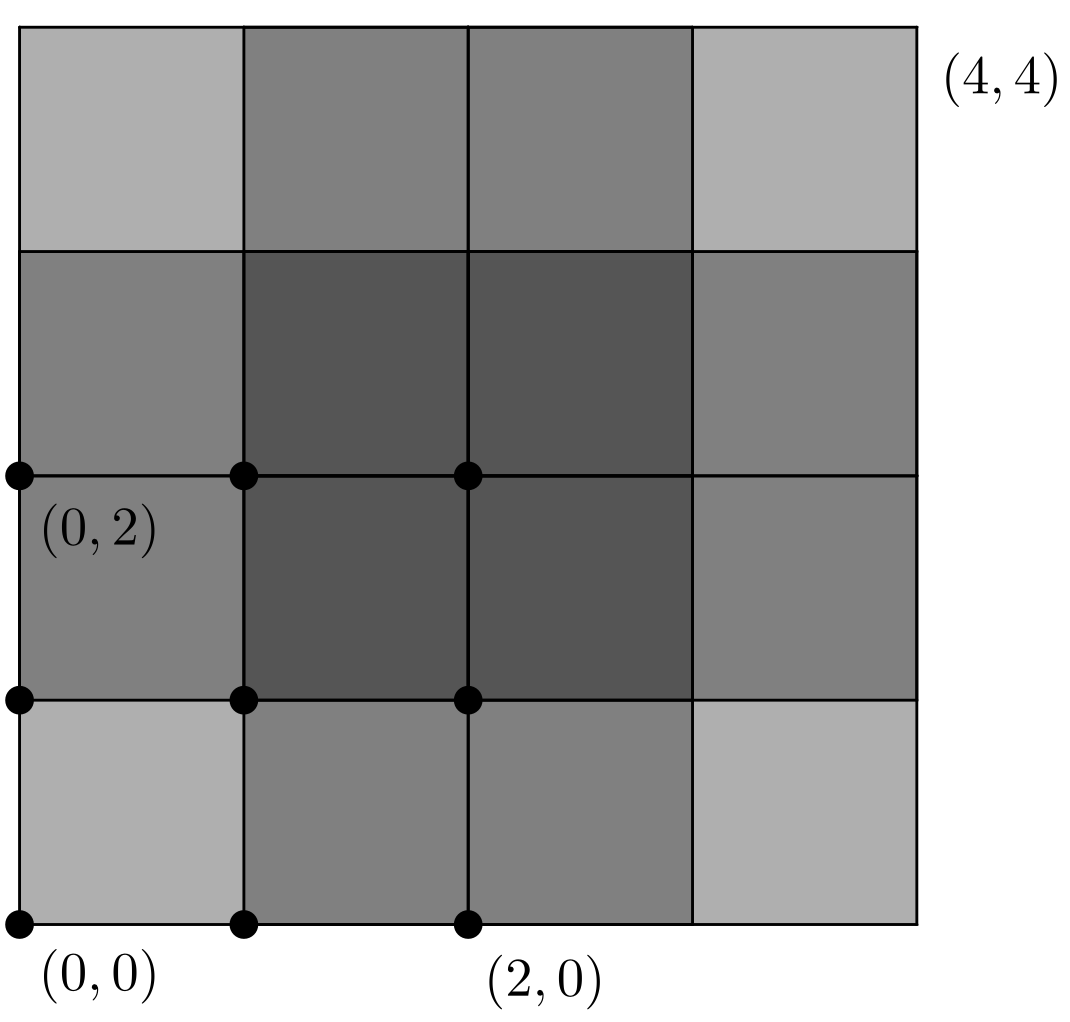}
        \caption{Depiction of Examples \ref{ex:1d} and \ref{ex:2d}.}\label{fig:IFS}
      \end{figure}

Generally, in dimension one the IFSs we consider are of the form
\begin{equation}\label{x99}
    \mathcal{S}=\{S_i(x)=x/L+t_i\}_{i=1}^M, \quad 2\leq L\in \mathbb{N}, \, t_i\in \mathbb{Q}.
\end{equation}
In dimension $d$ we are more restrictive as we require the following:

\begin{align}\label{x98}
    &\mathcal{S}:=\left\{
        S_i(\underline{x}):=\frac{1}{L}\underline{x}+\underline{t}_i
     \right\}_{i=1}^{M}
     , S_i:\mathbb{R}^{d}\to\mathbb{R}^{d},
    \notag \\
    &L\in\N \setminus \left\{0,1\right\},
    \, \underline{t}_i\in \mathbb{N}^{d},\, \exists h\in \mathbb{N},\,L-1|h,\,\forall j\in[d]_1\, 0=\min_{i}t_i(j), \, h=\max_{i}t_i(j),
    \end{align}
where $t_i(j)$ is the $j$-th coordinate of the $i$-th translation $\underline{t}_i$. One can easily construct such $d$-dimensional IFSs, a way to do this in dimension $2$ is explained in Figure \ref{fig:algo}. Here one should choose the contraction ratio $1/L$ and further a natural number $N\geq 2$. The importance of $N$ will be clear from Section \ref{detifs}.
\begin{figure}
        \centering
        \includegraphics[width=0.7\linewidth]{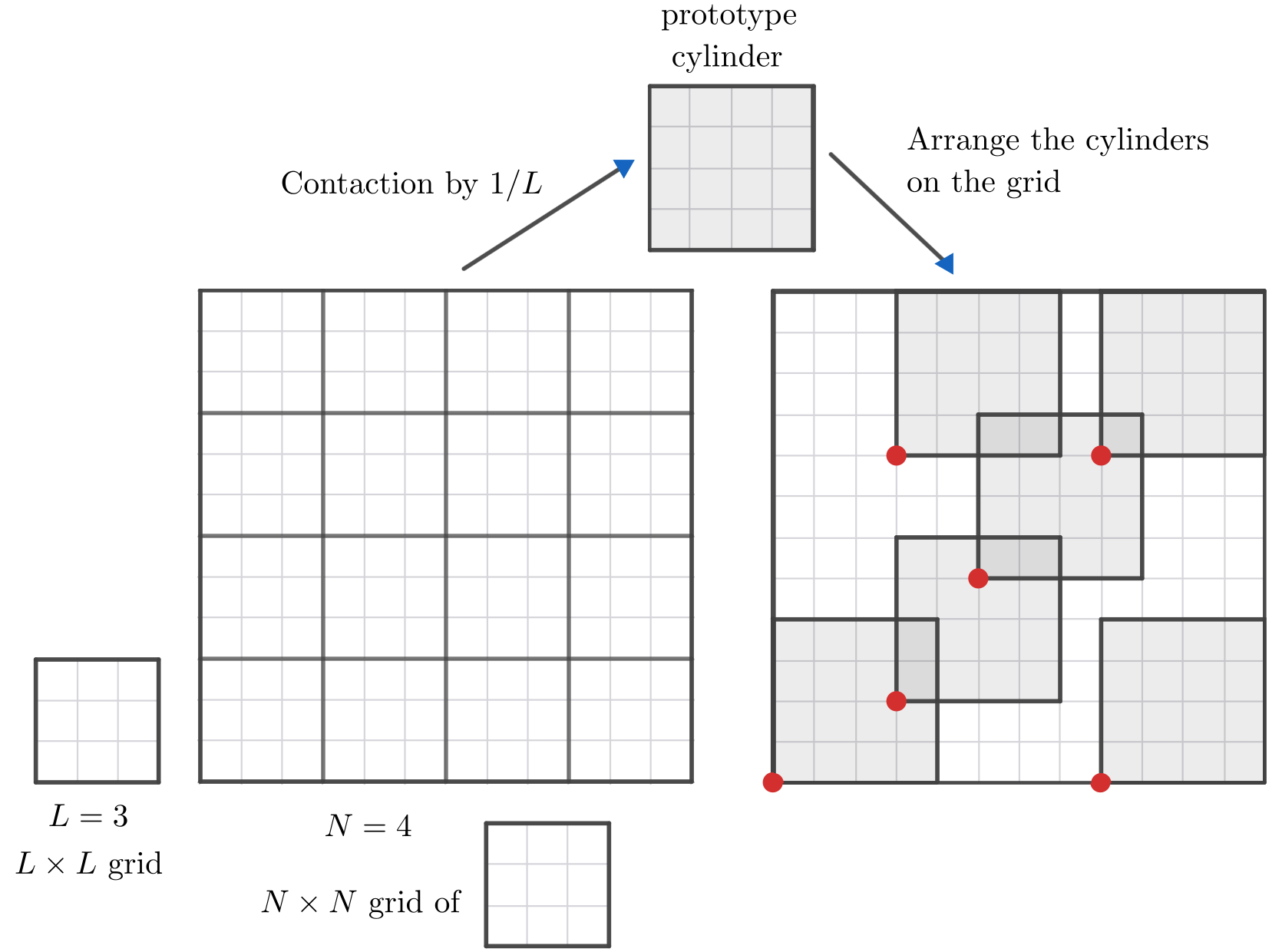}
        \caption{Depiction of the construction of an IFS satisfying \eqref{x98}. Here $L$ is the contraction ratio, and $N\in \mathbb{N}$ is chosen arbitrarily. We start with an $L\times L$ square, from this we form an $N\times N$ grid. We arrange the cylinders so that the grids are aligned. An example is shown in the right most image, where the cylinders are denoted with grey and their translations by red dots. The importance of $N$ is clear from Section \ref{detifs}.}\label{fig:algo}
      \end{figure}
    Throughout this note we call the one- and $d$-dimensional IFSs that satisfy \eqref{x99} and \eqref{x98} respectively \textit{integer self-similar IFS}s. 
    The overlapping structure of the cylinder sets of these IFSs can be described using certain matrices. When we consider the randomized attractor corresponding to the IFS, the overlapping structure can analogously be described using certain \textit{expectation matrices} (which we define in Section~\ref{sec:matrices}).

The dimension formula appears in terms of these expectation matrices under some positivity conditions imposed on the matrices.
Given the expectation matrices $\mathbf{M}_0, \dots, \mathbf{M}_{L^d-1}$ one can define the subadditive pressure:
\begin{equation}
    P(t)=P_{\mathbf{M}_0, \dots, \mathbf{M}_{L^d-1}}(t)
    := 
    \lim_{n\to\infty} 1/n \log(\sum_{\underline{\theta}\in \{0,\dots, L^d-1\}^n} \|\mathbf{M}_{\underline{\theta}}\|^t),
\end{equation}
where $\mathbf{M}_{\underline{\theta}}=\mathbf{M}_{\theta_1}\cdots \mathbf{M}_{\theta_n}$, for $\underline{\theta}=\theta_1\dots\theta_n$, and the matrix norm is $\|\|$, simply the 1-norm.

Assume that for all $\theta\in\{0,\dots, L^d-1\}$, $\mathbf{M}_{\theta}$ contains a strictly positive element in all rows and columns and there exists $\theta_1,\dots,\theta_n$ with $\theta_i\in \{0,\dots, L^d-1\}$ such that the product matrix $\mathbf{M}_{\theta_1}\dots \mathbf{M}_{\theta_n}$ is strictly positive. Let $\Lambda$ denote the random attractor. The dimension formula is given by 
\begin{equation}
    \dimH \Lambda= \inf_{t\in [0,1]}P(t)/\log(L),
\end{equation}
almost surely conditioned on non-extinction. 
For the precise definition of the random attractor see \eqref{a35}. For the precise statement see the main theorem of this paper: Theorem \ref{thm:dim}.

\subsection{Setup and notation}
 
In this note we confine ourselves to certain \textit {self-similar} IFSs, in dimension one these are of the form \eqref{x99} and in dimension $d$ they are conjugated via a similarity to \eqref{x98}. For a natural number $K>1$, we will often use the alphabets $\{0, \dots, K-1\}\eqqcolon[K]_0$ and $\{1,\dots, K\}\eqqcolon[K]_1$. Let $\Sigma_0^{[K]}=\{\emptyset\}$, where $\emptyset$ is the empty word.
For the finite alphabet $[K]$ (either $[K]_0$ or $[K]_1$) we use the following notation for $n\in \N$: 
\begin{equation*}
    \Sigma^{[K]}_n=[K] ^n, \quad \Sigma_*^{[K]}=\bigcup_{n=0}^{\infty} \Sigma^{[K]}_n, \quad \Sigma^{[K]}=[K]^{\N}.
\end{equation*} 
In natural language, these are the words of length $n$, the finite words, and the infinite words over the alphabet $[K]$ respectively. 
We also use the notational shorthand
\begin{equation*}
    S_{i_1\dots i_n}=S_{i_1}\circ\cdots\circ S_{i_n}.
\end{equation*}

For a $d$-dimensional ($d\geq 1$) self-similar IFS there exists a compact set $B\subset \R^d$, for which $S_i(B)\subset B$.
By a classical result of Hutchinson, for this IFS $\mathcal{S}$, (say, $\mathcal{S}=\{S_1,\ldots,S_M\}$) there exists a unique non-empty compact set satisfying 
\begin{equation*}
    \Lambda_{\mathcal{S}}=\bigcup_{i\in [M]_1} S_i(\Lambda_{\mathcal{S}}).
\end{equation*}
We refer to this set as the \textit{attractor}, and note that it can be equivalently defined as
\begin{equation*}
    \Lambda_{\mathcal{S}}=\bigcap_{n\in \N}\bigcup_{(i_1,\dots, i_n)\in [M]_1^n}S_{i_1\dots i_n}(B).
\end{equation*}

The definitions of the Hausdorff and box dimensions, the notion of conjugation of IFSs, and the terminology of symbolic spaces and trees are collected in the appendices.
\section{Model}

In this section, we introduce in detail the deterministic IFS underlying the random model which we will use throughout this note.
We are interested in IFSs which contain overlaps.
In general, overlaps lead to complex geometric structure which results in a wide range of complicated behaviour.
To be able to manage the overlaps, we restrict to a family of self-similar IFSs whose cylinders align with a fixed 
$L$-adic grid, and encode the overlaps by matrices recording how many cylinders of each type meet each grid cube.
This family of IFSs and their matrix representation was introduced and studied extensively by Ruiz in \cite{zbMATH05549434} where he considered self-similar measures on IFSs of this type.

\subsection{Deterministic integer self-similar IFSs}\label{detifs}

\subsubsection{Setup}
We fix an integer self-similar IFS $\mathcal{S}$, with common contraction ratio $1/L$, and $\#\mathcal{S}=M$.
Let $\zeta$ be the self-similar measure corresponding to the IFS $\mathcal{S}$ and the probability vector $(1/M, \dots, 1/M)$: for Borel sets $A$, the measure $\zeta$ satisfies
\begin{equation*}
    \zeta(A)=\sum_{j=1}^M \frac{1}{M} \zeta(S_j^{-1}A).
\end{equation*}
Next, we introduce a family of partitions (mod $0$) of $\mathbb{R}^d$ and refer to the elements of these partitions as \emph{$L$-adic cubes}:
\begin{multline*}
    \mathcal{D}_k\coloneqq
    \left\{
        \left[ (i_1-1)L^{-k},i_1 L^{-k} \right]\times \dots \times \left[ (i_d-1)L^{-k},i_d L^{-k} \right]:
        (i_1,\dots, i_d)=\underline{i}\in \mathbb{Z}^d
    \right\},\\
    k\in\left\{ -1,0,1,2,\dots \right\}.
\end{multline*}
Particular attention is given to those elements of $\mathcal{D}_{-1}$ which have positive $\zeta $-measure.
We call them \textit{basic cubes}.
We define the $N$ basic cubes as
\begin{equation*}
    \left\{J^{(u)}\coloneqq[b_u(1)L,(b_u(1)+1)L] \times \dots \times [b_u(d) L,(b_u(d)+1)L] \right\}_{u=1}^{N},
    \quad
    \underline{b}_u\in \mathbb{N}^d,\
    u=1,\dots ,N.
\end{equation*}
We call $u\in [N]_1$ (the cube $J^{(u)}$) a \emph{type}.
Here, $\{\underline{b}_u\}_{u=1, \dots, N}$ is arranged in a lexicographical order.
For $d=1$ the smallest interval that contains all the basic cubes is
\begin{equation*}
    I=[0,\text{Fix}(S_{M})]=\left[ 0,L\frac{t_{M}}{L-1} \right].
\end{equation*}
Here we used that $\frac{t_{M}}{L-1}\in\mathbb{N}$.

For $d>1$ it follows from condition~\eqref{x98} that the smallest cube containing all basic cubes is

\begin{equation}
    I=\Big[0,L\frac{h}{L-1}\Big]^d,
\end{equation}
similarly recall that $L-1$ divides $h$ according to our assumption.

For every $u\in [N]_1$ and $n\in\N$ the basic cube $J^{(u)}$ subdivides into $(L^{d})^{n}$ congruent subcubes of side-length $L^{-(n-1)}$ contained in $\mathcal{D}_{n-1}$, which we denote by $J _{\underline{\theta}}^{(u) }$, where $\underline{\theta }=(\theta _1,\dots ,\theta _n)\in [L^d]_0^n$.
We identify each $\theta\in[L^d]_0$ with $(\theta(1),\dots,\theta(d))\in[L]_0^d$ via $\theta=\sum_{j=1}^{d}\theta(j)L^{j-1}$. For $u\in[N]_1$ and $\underline{\theta}=(\theta_1,\dots,\theta_n)\in[L^d]_0^n$ let
\begin{equation*}
    J^{(u)}_{\underline{\theta}}\coloneqq\prod_{j=1}^{d}\Big[\,b_u(j)L+\sum_{\ell=1}^{n}\theta_\ell(j)L^{-(\ell-1)},\ b_u(j)L+\sum_{\ell=1}^{n}\theta_\ell(j)L^{-(\ell-1)}+L^{-(n-1)}\Big].
\end{equation*}

An example is shown in Figure \ref{fig:algo2}.
\begin{fact}\label{a96}
    For any $\underline{\ell}\in[M]_1^n$ and $ u\in [N]_1$ 
    \begin{enumerate}
        \item There exists $v\in[N]_1$ and $\widetilde{\underline{\theta}}\in [L^d]_0^n$ such that
        \begin{equation*}
            S_{\underline{\ell}}(J^{(u)})=J_{\widetilde{\underline{\theta}}}^{(v)}.
        \end{equation*}
        \item For $\underline{\theta}\in[L^d]_0^m$ 
        \begin{equation*}
            S_{\underline{\ell}}(J^{(u)}_{\underline{\theta}})=J_{\widetilde{\underline{\theta}}\underline{\theta}}^{(v)}.
        \end{equation*}
    \end{enumerate}
\end{fact}
For the proof see \cite[Fact 1.1]{orgovanyi_thesis}.
For every $\theta \in[L^d]_0$ we define the $N\times N$ matrix
\begin{equation}
    \label{a90}
    B_\theta (u,v) \coloneqq
    \#\left\{ \ell \in [M]_1:
        S_\ell (J^{(v)})=J _{\theta }^{(u) }
    \right\}.
\end{equation}
A depiction of this is shown in Figure \ref{fig:algo2}.
Further, for $\underline{\theta}\in [L^d]_0^n$ for $n>1$
\begin{equation*}
    B_{\underline{\theta}} (u,v) \coloneqq
    \#\left\{ \underline{\ell} \in [M]_1^n:
        S_{\underline{\ell}} (J^{(v)})=J _{\underline{\theta} }^{(u)}
    \right\}.
\end{equation*}
\begin{fact}\label{a97}
    For $\underline{\theta}=(\theta_1,\dots,\theta_n)\in [L^d]_0^n$ for $n>1$ we have 
    \begin{equation*}
        B_{\underline{\theta}}(u,v)= (\mathbf{B}_{\theta_1}\cdots\mathbf{B}_{\theta_n})(u,v).
    \end{equation*}
\end{fact}
The proof is presented in \cite[Fact 1.2]{orgovanyi_thesis}.
\begin{remark}
    In this note the matrices defined simultaneously for a letter $\theta\in [L^d]_0$ and for a word $\underline{\theta}\in[L^d]_0^n$ will satisfy the property in Fact~\ref{a97} that the matrix corresponding to a word $\underline{\theta}=\theta_1\dots \theta_n$ is equal to the product of matrices corresponding to the individual letters $\theta_1,\dots, \theta_n$.
    For this reason, it is safe to think about $\mathbf{A}_{\underline{\theta}}$ for some general family of matrices $\{\mathbf{A}_{\underline{\theta}}\}_{\underline{\theta}\in [L^d]_0^*}$ and the word $\underline{\theta}$ as the product $\mathbf{A}_{\theta_1}\cdots \mathbf{A}_{\theta_n}$.
\end{remark}
 We will use these matrices extensively in this note.

  \begin{figure}
        \centering
        \includegraphics[width=0.7\linewidth]{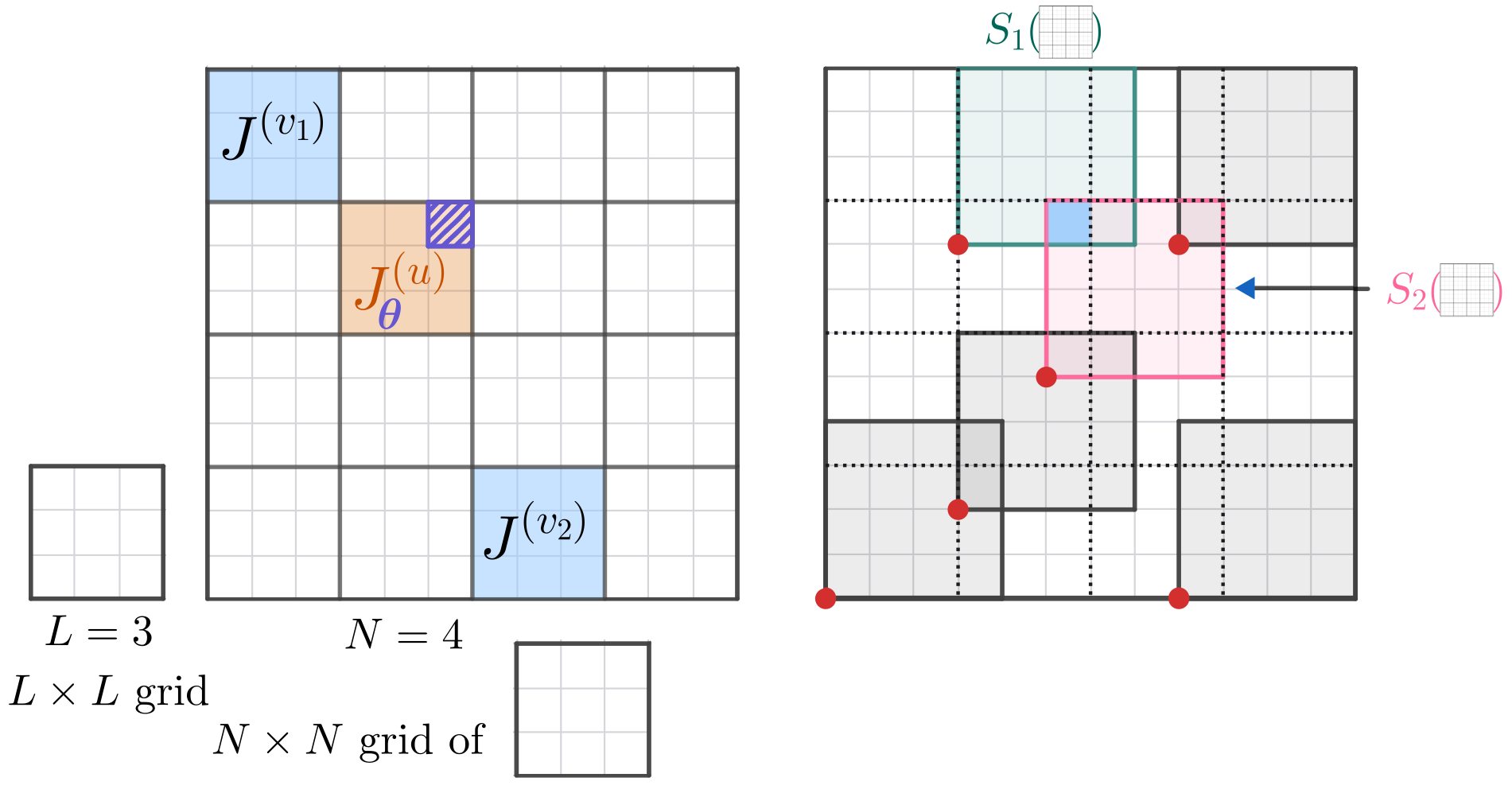}
        \caption{This figure is a continuation of Figure \ref{fig:algo}. This is a heuristic explanation for the construction of the matrices and the meaning of the types. The image shows an example type $J^{(u)}$, and $J^{(u)}_{\theta}$. Further in the example in the figure $S_1(J^{(v_2)})=S_2(J^{(v_1)})=J^{(u)}_{\theta}$. }\label{fig:algo2}
      \end{figure}

\section{Substitution random sets.}\label{a59}
In this section we introduce a model to randomize the above-defined integer self-similar IFSs, resulting in a random attractor.
Given a general IFS $\mathcal{S}=\{S_i(x)\}_{i=1}^M$, with $M$ maps, one can define a labelled Galton--Watson tree as a subset of a full $M$-ary tree (see Section \ref{app:symbolictrees}).
The approximations of this random tree are projected to $\mathbb{R}^d$ to form the approximations of the attractor as we explain below.

In the construction we assign a subset $\omega_{i_1\dots i_n}\subset [M]_1$ to \textit{each} node $i_1\dots i_n$ of $\mathcal{T}_M$, indicating which children survive, even if the parent itself is not retained. Survival is defined inductively: the root survives, and a level-$n$ node survives if its parent survives and the set assigned to the parent indicates so.

Fix an IFS $\mathcal{S}$ with $\#\mathcal{S}=M\geq 2$. 
Let $\Omega= \mathcal{P}([M]_1)$, the power set of $[M]_1$. Let $\mathcal{F}=\mathcal{P}(\mathcal{P}([M]_1))$ denote the usual $\sigma$-algebra on the finite set $\mathcal{P}([M]_1)$.
Next, fix a probability vector $\underline{p}=(p_h)_{h\in \Omega}$, that is, $p_h\geq 0$ and $\sum_{h\in \Omega}p_h=1$.
We define a probability measure $\mathbb{P}$ on $\Omega$ given for $A\in \mathcal{F}$ by
\begin{equation*}
  \mathbb{P}(A)=\sum_{h\in A}p_h.
\end{equation*}
This is the base probability space $(\Omega, \mathcal{F}, \mathbb{P})$.

Next, we consider the $M$-ary tree $\mathcal{T}_M$ and the product space 
\[
    \left(
        \Omega^{\mathcal{T}_M},
        \ \mathcal{F}^{\mathcal{T}_M},
        \ \mathbb{P}^{\mathcal{T}_M}
    \right)
    =
    \left(
        \prod_{{\underline i}\in \mathcal{T}_M} \Omega_{\underline i},\ 
        \bigotimes_{{\underline i}\in \mathcal{T}_M} \mathcal{F}_{\underline i},\ 
        \bigotimes_{{\underline i}\in \mathcal{T}_M} \mathbb{P}_{\underline i}
    \right),
\]
where $(\Omega_{\underline i}, \mathcal{F}_{\underline i}, \mathbb{P}_{\underline i})$ are independent copies of the base space $(\Omega, \mathcal{F}, \mathbb{P})$. 
Let $\underline{W}\colon\Omega\to \{0, 1\}^M$ denote the random variable
\begin{equation*}
    \underline{W}(\omega)=(\mathds{1}\{1\in \omega\}, \mathds{1}\{2\in \omega\}, \dots, \mathds{1}\{M \in \omega\}).
\end{equation*}
To each level-$n$ node $i_1,\dots, i_n\in \mathcal{T}_M$ we assign a copy of the vector random variable $\underline{W}^{[i_1\dots i_n]}=(W^{[i_1\dots i_n]}_{1},\dots, W^{[i_1\dots i_n]}_{M})$, a copy of $\underline{W}$.
Since the probability space is a product space and all random variables $\underline{W}^{[i_1\dots i_n]}$ depend only on the coordinate ${\underline i}=i_1\dots i_n$, the family of random variables $\{\underline{W}^{[i_1\dots i_n]}\}_{\underline{i}\in \mathcal{T}_M}$ is jointly independent.
We define another random variable for $\underline{i}\in \mathcal{T}_M$,
\begin{equation*}
    Q_{\underline i}=W^{[\emptyset]}_{i_1} W^{[i_1]}_{i_2}\dots W^{[i_1\dots i_{n-1}]}_{i_n}.
\end{equation*}
In this way, we can define the symbolic set of retained cylinders,
\begin{equation*}
    \mathcal{E}_n\coloneqq\{{\underline i}\in [M]_1^n: Q_{\underline i}=1\}.
\end{equation*}

The $n$-th approximation of the random set and the attractor are given by
\begin{equation}\label{a35}
    \Lambda_n=\bigcup_{{\underline i}\in \mathcal{E}_n}S_{\underline i}(I),\quad \Lambda=\bigcap_{n\in \N}\Lambda_n.
\end{equation}

Furthermore, we use the notation:
\begin{equation*}
    \mathcal{E}_{\infty}=\{\mathbf{i}\in [M]_1^{\N}:\: \forall n \in \N, \:\mathbf{i}|_n\in \mathcal{E}_n\}, 
\end{equation*}
and in this case the \emph{substitution random set} is
\begin{equation*}
    \Lambda=\pi(\mathcal{E}_\infty),
\end{equation*}
where $\pi$ is the natural projection corresponding to the IFS $\mathcal{S}$. The details are explained in Section \ref{survival}. We will assume throughout the note that $\mathbb{E}(\#\mathcal{E}_1)>1$. This guarantees that $\Lambda\neq \emptyset$ and $\mathcal{E}_{\infty}\neq \emptyset$ with positive probability. We call the event $\Lambda\neq \emptyset$ \emph{non-extinction}.

In what follows we will also consider the random (labelled) tree corresponding to the random sets $\mathcal{E}_n$. Throughout this note we call this \emph{labelled Galton–Watson tree}. Namely, let $\mathcal{T}$ be the random subtree of $\mathcal{T}_M$ whose vertices at depth $n$ are exactly the elements of $\mathcal{E}_n$.

\subsection{A special case: the coin-tossing model}
The \emph{coin-tossing model} is a special case of the substitution model.
Consider again an IFS $\mathcal{S}$ with $\#\mathcal{S}=M$.
The distribution $\mathbb{P}$ on $H=\mathcal{P}([M]_1)$ is the following: 
\begin{equation*}
    p_h\coloneqq p^{\#h}(1-p)^{M-\#h}, \quad h\in H.
\end{equation*}
We refer to this as the coin-tossing model, as the retention of the cylinders is determined via independent coin tosses.

The examples appearing in this note mostly correspond to the coin-tossing model.
One reason for this is that the coin-tossing model already shows a colourful palette of behaviour.
Another reason is that the coin-tossing model is simpler to work with since the expectation matrices (which we formally define in Section \ref{sec:matrices}) for $\theta\in[L^d]_0$ and $\underline{\theta}\in [L^d]_0^n$ can be written as
\begin{equation*}
    \mathbf{M}_{\theta}=p\cdot \mathbf{B}_{\theta},\, \mathbf{M}_{\underline{\theta}}=p^n\mathbf{B}_{\underline{\theta}}.
\end{equation*}
Here, we recall that $B_\theta (i,k) =
\#\left\{ \ell \in [M]_1:
S_\ell (J^{(k)})=J _{\theta }^{(i) }
\right\}$ was defined in \eqref{a90}.

The random attractor in the coin-tossing case will usually be denoted by $\Lambda_{p}$ to emphasize the dependency on the probability parameter $p$.

\section{Expectation matrices and the pressure function}\label{sec:matrices}
The overlapping structure of the random set is encoded by matrices which count, in expectation, how many cylinders of a given type a given $L$-adic cube contains.
For $\theta \in[L^d]_0$ we define the $N\times N$ \emph{expectation matrix} $\mathbf{M}_\theta$ by
\begin{equation*}
    M_\theta (u,v) \coloneqq
    \mathbb{E}\left[\#\left\{ \ell \in \mathcal{E}_1:
        S_\ell (J^{(v)})=J _{\theta }^{(u) }
    \right\}\right]
    =
    \sum_{\substack{\ell\in[M]_1\\ S_\ell (J^{(v)})=J _{\theta }^{(u)}}} \mathbb{P}(\ell\in \mathcal{E}_1).
\end{equation*}
Clearly, $
    \mathbb{P}(\ell\in\mathcal{E}_1)=\sum_{\ell \in h}p_h$.
More generally, for $\underline{\theta}\in [L^d]_0^n$ we set
\begin{equation}\label{eq:expmatrixword}
    M_{\underline{\theta}} (u,v) \coloneqq
    \mathbb{E}\left[\#\left\{ \underline{\ell} \in \mathcal{E}_n:
        S_{\underline{\ell}} (J^{(v)})=J _{\underline{\theta} }^{(u)}
    \right\}\right].
\end{equation}
\begin{fact}
    For $\underline{\theta}=(\theta_1,\dots,\theta_n)\in [L^d]_0^n$ the matrices satisfy
    $\mathbf{M}_{\underline{\theta}}=\mathbf{M}_{\theta_1}\cdots\mathbf{M}_{\theta_n}$.
\end{fact}
For the proof see \cite[Fact 1.4]{orgovanyi_thesis}.
Here and throughout, $\|\cdot\|$ denotes the $1$-norm of a matrix, that is the sum of the absolute values of its entries.

\begin{ass}[Jointly primitive]\label{ass:pos}
    The family $\mathcal{M}=\{\mathbf{M}_\theta\}_{\theta\in[L^d]_0}$ of expectation matrices satisfies:
    \begin{enumerate}[label=\textbf{(P\arabic*)}]
        \item for every $\theta\in [L^d]_0$ the matrix $\mathbf{M}_{\theta}$ is \emph{allowable}, namely it contains a strictly positive
              entry in every row and in every column;
        \item there exist $n\in\N$ and $\theta_1,\dots,\theta_n\in[L^d]_0$ such that the product
              $\mathbf{M}_{\theta_1}\cdots \mathbf{M}_{\theta_n}$ is strictly positive.
    \end{enumerate}
\end{ass}

\begin{definition}[Pressure function]\label{def:pressure}
    Let $\mathbf{M}_0, \dots, \mathbf{M}_{L^d-1}$ be the expectation matrices of a substitution random set. We assume that each of them is non-negative and are jointly primitive (see Assumption \ref{ass:pos}).
    Then the pressure function corresponding to $\mathcal{M}=\{\mathbf{M}_{0}, \dots, \mathbf{M}_{L^d-1}\}$ is
    \begin{equation}\label{y99}
        P(t)\coloneqq \lim_{n\to \infty} \frac{1}{n} \log \left(\sum_{\underline{\theta}\in [L^d]_0^n}\|\mathbf{M}_{\underline{\theta}}\|^t\right).
    \end{equation}
\end{definition}
The existence of the limit, the convexity and the smoothness of $P$ are known in the literature and are recalled in Section~\ref{a46}.

\section{The main theorem}

\begin{theorem}[Main theorem]\label{thm:dim}
    Let $\mathcal{S}$ be an integer self-similar IFS on $\R^d$ conjugate by a similarity to an IFS of the form \eqref{x99} for $d=1$
    and \eqref{x98} for $d>1$, with common contraction ratio $1/L$ and $\#\mathcal{S}=M$.
    Let $\Lambda$ be the substitution random set corresponding to $\mathcal{S}$ and to an offspring
    distribution $\underline{p}=(p_h)_{h\subset [M]_1}$, and let $P$ be the associated pressure function.
    If Assumption~\ref{ass:pos} holds, then almost surely conditioned on non-extinction
    \begin{equation}\label{a42}
        \dimH(\Lambda)=\dimB(\Lambda)=\dimP(\Lambda)=\inf_{t\in[0, 1]}\frac{P(t)}{\log(L)}.
    \end{equation}
\end{theorem}
Recall that non-extinction is the event $\Lambda\neq \emptyset$. This happens with positive probability provided that $\sum_h p_h \#h = \mathbb{E}(\#\mathcal{E}_1)>1$, which we always assume for the rest of the note.

\begin{application}[Application to Example \ref{ex:1d}]
    We consider the coin-tossing randomization, the retention probability is $p$. The right angle Sierpi\'nski carpet as well as its projections are non-empty iff $p>1/3$, and its dimension is almost surely $\log_2(3p)$ (on non-extinction).
Consider $\proj: \R^2\to\R$,
        $\proj(x,y)\coloneqq-x+y$,
    the rescaled version of the $45$-degree projection. Then the rescaled projected IFS  (see Figure \ref{fig:IFS}) is
    $\mathcal{S}=\left\{ S_i(x)=\frac{1}{2}x+i \right\}_{i=0}^2 $. In this case $L=N=2$. The types are determined by the basic cubes,
    $J^{(1)}\coloneqq\left[ 0,2 \right]$ and $J^{(2)}\coloneqq[2,4]$.

The corresponding matrices are:

\[
        \mathbf{B}_0=
        \begin{bmatrix}
        1 & 0 \\
        1 & 1  \\

        \end{bmatrix},\quad
        \mathbf{B}_1=
        \begin{bmatrix}
        1 & 1 \\
        0 & 1  \\

        \end{bmatrix}.
\]

Its dimension theoretic properties are collected here:
\begin{itemize}
    \item It follows from the work of Pollicott and Vytnova\footnote{This is based on personal communication.}, that for $p\leq 0.672$ the Lebesgue measure of the random attractor is almost surely $0$, however for $p\geq 0.673$ the Lebesgue measure is positive almost surely conditioned on non-extinction. From this it follows that if $p_1$ denotes the critical probability so that for all $p\geq p_1$ $\dim_H \Lambda_p=1$, then $0.672\leq p_1\leq 0.673$.
    \item It follows from \cite{zbMATH06379809} together with Theorem \ref{thm:dim} that there exists a parameter interval where $\dimH \Lambda_p<\dimS \Lambda_p =\log(3p)/\log(2)<1$. Here $\dimS$ stands for the expected similarity dimension, see \eqref{dimsim}.
\end{itemize}
\end{application}

\begin{application}[Application on Example \ref{ex:2d}]
    In this case $N=4$, and the corresponding matrices are:
    \begin{align*}
        \mathbf{B}_0=
            \begin{bmatrix}
             1 & 0 & 0 & 0 \\
             1 & 1 & 0 & 0 \\
             1 & 0 & 1 & 0 \\
             1 & 1 & 1 & 1 \\
            \end{bmatrix},\,
        \mathbf{B}_1=
                \begin{bmatrix}
                 1 & 1 & 0 & 0 \\
                 0 & 1 & 0 & 0 \\
                 1 & 1 & 1 & 1 \\
                 0 & 1 & 0 & 1 \\
                \end{bmatrix},\, \\
        \mathbf{B}_2=
                \begin{bmatrix}
                    1 & 0 & 1 & 0 \\
                    1 & 1 & 1 & 1 \\
                    0 & 0 & 1 & 0 \\
                    0 & 0 & 1 & 1 \\
                   \end{bmatrix},\,
                   \mathbf{B}_3=
                   \begin{bmatrix}
                    1 & 1 & 1 & 1 \\
                    0 & 1 & 0 & 1 \\
                    0 & 0 & 1 & 1 \\
                    0 & 0 & 0 & 1 \\
                   \end{bmatrix}.
    \end{align*}
    Its dimension theoretic properties are as follows:
    \begin{itemize}
        \item When $p>0.7712$ then the set has positive two-dimensional Lebesgue measure almost surely conditioned on non-extinction (this is discussed in \cite{zbMATH08181170}).

    \end{itemize}
\end{application}

\section{Dimension theory of random self-similar sets, interpretation of the formula}
\subsection{Historical overview}
We begin this section with a brief historical overview of the dimension theory of certain statistically self-similar sets. 

When analyzing rational projections or slices of sponges, it is natural to subdivide the cylinders into shapes and use matrices to track their evolution inside columns. For instance, Manning and Simon \cite{zbMATH06152557} considered rational slices of the Sierpi\'nski carpet. Slices of more general $2$-dimensional carpets appeared in \cite{zbMATH06379809}; the authors used different matrices to encode the evolution of shapes. In this note we follow the general framework used by Ruiz (see \cite{zbMATH05549434}) to encode the growth of shapes.
In the random setting shapes appeared in \cite{zbMATH05255663,zbMATH05643268}, when studying differences of random Cantor sets. 

Our starting point is the dimension formula for boundaries of Galton--Watson trees, established by Hawkes in 1981 \cite{zbMATH03734900}.
His result can also be used to determine the dimension of certain statistically self-similar sets obtained by projecting the Galton--Watson tree to the attractor of a self-similar IFS, as described in Section~\ref{a59}.
In order for this result to be applicable, the IFS has to be homogeneous and has to satisfy a relatively strong separation condition. 
From the perspective of this note, the most important family of IFSs that satisfy these is the one defining a $d$-dimensional carpet.
Let $Z$ denote the offspring random variable of the Galton--Watson tree.
Then for a random carpet with contraction $1/L$,
\begin{equation}
    \dimH \Lambda = \frac{\log(\mathbb{E}[Z])}{\log(L)}
\end{equation}
almost surely conditioned on non-extinction. 
We note that the dimension formula follows from the formula in more general constructions, such as the ones appearing in the work of Falconer \cite{zbMATH04011530} or of Mauldin and Williams \cite{mauldin1986random}.
For an exposition, see \cite[Chapter 15]{zbMATH06653785}. 

In \cite{zbMATH04115622,zbMATH04022321}, the authors further permit the following: for the deterministic skeleton, two first level cylinders are either identical or their interior is disjoint.
These IFSs can be viewed as projections of carpets onto coordinate subspaces.
The dimension formula is already much more complicated in these cases. 

Consider a $d$-dimensional carpet with contraction ratio $1/L$, and let $\proj$ denote the projection to the $k$-dimensional subspace spanned by the first $k$ coordinate vectors.
The projected set is contained in $\bigcup_{i=0}^{L^k-1} D_i$, where each $D_i$ is a closed $L$-adic cube in dimension $k$.
The ``expected column sums'' are given by $m_i\coloneqq\mathbb{E}(\#\mathcal{E}^i_1)$, where $\mathcal{E}^i_1=\{j\in \mathcal{E}_1: \proj(S_j([0, 1]^d))=D_i\}$, for $i\in \{0, \dots, L^k-1\}$. 

For simplicity, we assume that $k=1$ and $m_i>0$ for all $i\in [L]_0$.
In this case, the formula for the almost sure dimension is given by
\begin{equation*}\label{x111}
    \dimH(\proj (\Lambda))=\inf_{t\in[0,1]}\frac{\log(\sum_{i=0}^{L-1} m_i^t)}{\log (L)}.
\end{equation*}

We assume that not all the $m_i$ are equal.
In this case, the function $\phi(t)=\sum_{i=0}^{L-1} m_i^t$ is strictly convex and continuous.
It has a unique minimum somewhere in $[-\infty,\infty]$.
If the minimum is attained in $[-\infty, 0]$, the dimension is the ambient dimension (in this case $1$).
If it is attained in $[1, \infty]$, the almost sure dimension is the dimension of the original set.
Otherwise, the minimum is strictly less than $\phi(1)$; hence the dimension drops. 

It is clear now that the dimension formula \eqref{a42} given in terms of the pressure function (see \eqref{y99}) is the matrix analogue of \eqref{x111}. The pressure function is the matrix analogue of $\log(\sum_{i=0}^{L-1} m_i^t)$. We remark that the reason why we need to use matrix analogues is because of the more complicated overlapping structure appearing in these cases of IFSs.

\section{Preliminaries}\label{a46}
For non-negative numbers $x$ and $y$, we write $x \lesssim_{a,b} y$, $x \gtrsim_{a,b} y$ 
if there exists a constant $C>0$, depending only on the parameters $a,b$, such that
$x \leq C\,y$, $x \geq C\,y$ respectively. We write $x \approx_{a,b} y$ if both
$x \lesssim_{a,b} y$ and $x \gtrsim_{a,b} y$ hold. 
Let $Z$ be the offspring random variable of the underlying Galton--Watson tree.
In what follows, the quantity
\begin{equation}\label{dimsim}
    \dimS(\Lambda)=\frac{\log(\mathbb{E}[Z])}{\log(L)}
\end{equation}
is called the \emph{expected similarity dimension} or the \emph{similarity dimension of the random system}.
Throughout this section, we use the metric $d(\mathbf{i}, \mathbf{i})=0$ and $d(\mathbf{i},\mathbf{j})=L^{-|\mathbf{i}\wedge \mathbf{j}|}$ for $\mathbf{i}\neq\mathbf{j}$ on the symbolic space $\Sigma^{[L^d]_0}$.
Here $|\mathbf{i}\wedge \mathbf{j}|$ is the length of the common prefix of $\mathbf{i}$ and $\mathbf{j}$, i.e.\ the largest $n\geq 0$ such that $i_k=j_k$ for all $k\leq n$.
The relevant properties of the pressure function $P(t)$ are collected by Bárány and Rams in \cite{zbMATH06379809}, and were originally proven by Feng in \cite{Feng_2007} and Feng and Lau in \cite{zbMATH01804061}.
\begin{lemma}\label{b88}
    The key properties for us are the following:
\begin{itemize}
    \item The limit exists for all $t$.
    \item The pressure function is convex for all $t$.
    \item The pressure function is continuously differentiable for all $t>0$.
\end{itemize}
\end{lemma}

\begin{remark}
   Note that in \cite{zbMATH06379809} the pressure function is normalized by dividing by the constant $\log(L)$.
   We omit this normalization. 
\end{remark}
Another important quantity is the Lyapunov exponent.
\begin{definition}\label{def:lyap}
For an ergodic measure $\nu$ on $\Sigma^{[L^d]_0}$ the \emph{Lyapunov exponent} of $\mathcal{M}$ with respect to $\nu$ is
\begin{equation*}
    \lambda_\nu(\mathcal{M})\coloneqq\lim_{n\to\infty}\frac{1}{n}\log\|\mathbf{M}_{\pmb\theta|_n}\| \quad\text{for $\nu$-almost every } \pmb\theta\in\Sigma^{[L^d]_0}.
\end{equation*}
By (P1) the products $\mathbf{M}_{\pmb\theta|_n}$ are non-zero, and since $\|\cdot\|$ is submultiplicative on non-negative matrices, the limit exists and is $\nu$-almost surely constant by Kingman's subadditive ergodic theorem. When the ergodic measure is the uniform measure, we simply write $\lambda$ for the Lyapunov exponent.
\end{definition}
\begin{lemma}[{\cite[proof of Lemma 4.9]{zbMATH06379809}}]\label{x46}
    The pressure satisfies $\lim_{t\to 0+}P'(t)=\lambda$.
\end{lemma}
Throughout this section, we will require the following key lemma which was originally proved in \cite{zbMATH01804061} by Feng and Lau.
The lemma appeared in a similar form in \cite{zbMATH06379809} without proof. 

\begin{lemma}[{\cite[Lemma 4.5]{zbMATH06379809}}]\label{gibbsm}
    For every $t>0$, there is a unique ergodic, left-shift-invariant Gibbs measure $\mu_t$ on $\Sigma^{[L^d]_0}$ with the following properties: for any $\pmb{\theta}=\left(\theta_1, \theta_2,\dots \right) \in \Sigma^{[L^d]_0}$, $n\in \N$:
    \begin{equation}
        \mu_t\left(\left[\pmb{\theta}|_n\right]\right) \approx_{t, \mathcal{S}} \|\mathbf{M}_{\pmb{\theta}|_n}\|^t \exp(-n P(t)),
    \end{equation}

    \begin{equation}
        \operatorname{dim}_H \mu_t=\frac{-t P^{\prime}(t)+P(t)}{\log(L)},
    \end{equation}
    and
    \begin{equation}\label{a24}
        \lambda_{\mu_t}(\mathcal{M})=P^{\prime}(t).
    \end{equation}
\end{lemma}
Recall that $\lambda_{\mu_t}(\mathcal M)$ is as in Definition~\ref{def:lyap} the \emph{Lyapunov exponent} with respect to the ergodic measure $\mu_t$.

Another important value (other than $\lambda$) is $\eta\coloneqq P'(1)$. From convexity and Lemma \ref{x46}, it follows that $\lambda\leq \eta$, and therefore we have the following ``Phases'':
\newline 
\begin{enumerate*}[label=(PH\arabic*), itemjoin={\quad}]
    \item $\lambda>0$,\label{ph1}
    \item $\lambda=0$, $\eta>0$,\label{ph2}
    \item $\lambda=0$, $\eta=0$,\label{ph3}
    \item $\lambda<0$, $\eta>0$,\label{ph4}
    \item $\lambda<0$, $\eta=0$,\label{ph5}
    \item $\lambda<0$, $\eta<0$.\label{ph6}
\end{enumerate*}
\newline
We note that for Phase \ref{ph1} the random attractor has positive Lebesgue measure, consequently, its dimension equals the ambient dimension.
\section{The measure of the random attractor}
The proof for certain phases depend on the dimension of certain measures. For this reason we cite our earlier theorem on the positivity of measures of the attractor.

Before we state the theorem we introduce an additional regularity condition, to treat the critical case.

For $n\in\N$, $\underline\theta\in[L^d]_0^n$ and $v\in[N]_1$ let $\underline{Y}(\underline{\theta}, v)=(Y(\underline{\theta}, v)(w))_{w\in [N]_1}$ be a vector-valued random variable such that $Y(\underline\theta, v)(w)$ has the same distribution as $\#\{\underline i\in \mathcal{E}_n: S_{\underline i}(J^{(w)})=J^{(v)}_{\underline\theta}\}$.

\begin{condition}\label{a41}
For a fixed ergodic measure $\nu$ on $\Sigma^{[L^d]_0}$ there exist $n\in\N$ and $\underline\theta\in [L^d]_0^n$ such that $\nu([\underline\theta])>0$ and for each $v\in[N]_1$
\begin{equation*}
    \mathbb{P}\big(\underline{Y}(\underline\theta, v)\in \{\underline{0},\underline e_1,\dots,\underline e_N\}\big)<1,
\end{equation*}
where $\underline e_j$ denotes the $j$-th unit vector of $\R^N$.
\end{condition}

\begin{remark}
    Condition~\ref{a41} is equivalent to a strong regularity property of multitype branching processes in random environments introduced by Tanny in \cite{zbMATH03738673}. 
\end{remark}
\begin{remark}\label{rem:cointossing-regular}
For coin-tossing systems, Condition~\ref{a41} holds for the uniform measure. Indeed, by Assumption~\ref{ass:pos} there exists a word $\underline\theta$ such that $\mathbf{B}_{\underline\theta}$ is strictly positive, hence every row sum of $\mathbf{B}_{\underline\theta}$ is at least $N\geq 2$.
\end{remark}
For $u\in[N]_1$ let $\Gamma_u\colon[0,1]^d\to J^{(u)}$, $\Gamma_u(\underline x)=L(\underline x+\underline b_u)$, and let $\Pi_L\colon\Sigma^{[L^d]_0}\to[0,1]^d$ be the $L$-adic coding map, consistent with the indexing of the cubes $J^{(u)}_{\underline\theta}$. For a measure $\nu$ on $\Sigma^{[L^d]_0}$ we write
\begin{equation*}
    \widetilde{\nu}\coloneqq\sum_{u=1}^{N}(\Gamma_u\circ\Pi_L)_*\nu .
\end{equation*}

Recall that $\lambda$ is the Lyapunov expoentn corresponding tot he matrices when $\nu$ is the uniform measure on $\Sigma^{[L^d]_0}$. In this case, equivalently,
\begin{equation*}
    \lambda=\lim_{n\to \infty} \frac{1}{n}\sum_{\underline{\theta}\in [L^d]_0^n} \frac{1}{L^{dn}}\log(\|\mathbf{M}_{\underline{\theta}}\|).
\end{equation*}
    \begin{theorem}[{\cite[Theorem 4.3]{zbMATH08181170}}]\label{measureofattr_theorem}
    Consider an ergodic measure $\nu$ on $\Sigma^{[L^d]_0}$. Assume that $\mathcal{M}$ satisfies (P1), and that (P2) holds with a word satisfying $\nu([\theta_1\dots\theta_n])>0$. Then:
    \begin{itemize}
        \item If $\lambda_{\nu}(\mathcal{M})>0$, then $\mathbb{P}(\widetilde{\nu}(\Lambda)>0)>0$. If moreover $\nu$ is the uniform measure or one of the Gibbs measures $\mu_t$ of Lemma~\ref{gibbsm}, then $\widetilde{\nu}(\Lambda)>0$ almost surely conditioned on non-extinction.
        \item If $\lambda_{\nu}(\mathcal{M})<0$ then $\widetilde{\nu}(\Lambda)=0$ almost surely.
        \item If $\lambda_{\nu}(\mathcal{M})=0$ and the vector random variables $\{\underline{Y}(\theta, v)\}$ corresponding to the random attractor satisfy Condition~\ref{a41} then $\widetilde{\nu}(\Lambda)=0$ almost surely.
    \end{itemize}
\end{theorem}
We remark that in \cite[Theorem 4.3]{zbMATH08181170} the theorem is stated for $d$-dimensional coin-tossing systems and this is stated in \cite{orgovanyi_thesis} for the more general substitution 1-dimensional sets, but it easily translates to $d$-dimensional substitution sets, analogously to how it is explained in the Appendix of \cite{zbMATH08181170}.
\section{Meaning and implications of Theorem~\ref{thm:dim}}
\subsection{Analysis of the formula in dimension 1}
Generally, the formula given in Theorem~\ref{thm:dim} is hard to approximate. Here, we summarize what we know about this formula in the case of the coin-tossing model. In what follows we describe the function $p\to \dimH \Lambda_p$, where $\Lambda_p$ is the projected (1-dimensional) random attractor corresponding to the coin-tossing system. We give a brief description and a sketch illustration in Figure \ref{a22}. For coin-tossing systems we use the deterministic matrices $\mathbf{B}_0,\dots,\mathbf{B}_{L-1}$ from \eqref{a90}. Let $\widehat{P}$ denote the pressure function \eqref{y99} of $\{\mathbf{B}_0,\dots,\mathbf{B}_{L-1}\}$, and let
\begin{equation*}
    \widehat{\lambda}\coloneqq\lim_{t\to0+}\widehat{P}'(t),\qquad \widehat{\eta}\coloneqq\widehat{P}'(1)
\end{equation*}
(these exist by Lemma~\ref{b88}). Since $\mathbf{M}_\theta=p\,\mathbf{B}_\theta$, the pressure of the expectation matrices is $P(t)=\widehat{P}(t)+t\log p$, hence $\lambda=\widehat{\lambda}+\log p$ and $\eta=\widehat{\eta}+\log p$, where $\lambda$ and $\eta$ are defined in Section~\ref{a46}. In particular, $\lambda>0$ if and only if $p>\mathrm{e}^{-\widehat{\lambda}}$, and $\eta>0$ if and only if $p>\mathrm{e}^{-\widehat{\eta}}$. When the matrices are numbers, $\widehat{P}(t)=\log\big(\sum_i \widehat{m}_i^{\,t}\big)$, where $\widehat{m}_i$ is the number of cylinders in the $i$-th column.

In the case where $\lambda\geq 0$, the infimum is attained at $t=0$.
In this case, the dimension is the ambient dimension.
If $\eta\leq 0$, the infimum is attained at $t=1$, hence the dimension is $P(1)/\log(L)=\log\left(\sum_{h\in H}p_h\cdot \#h\right)/\log(L)$, which is the expected similarity dimension of the IFS. Otherwise, the infimum is attained in the interior of $[0,1]$ hence the dimension strictly drops from the expected similarity dimension. 

For coin-tossing systems with probability parameter $p$ as we increase $p$ from $1/M$ to $1$ the system goes through the upper mentioned three phases.

\begin{itemize}
    \item \emph{when the matrices are numbers:} this is the case when one projects a sponge to the coordinate axis, which was described in the previous subsection. In this case the formula is provided by \cite{zbMATH04022321,zbMATH04115622}. Unless all $\widehat m_i$ agree, the function $p\mapsto \dimH \Lambda_p$ is concave and continuous on $(1/M,1]$, and
real-analytic on each of the intervals $(1/M,\mathrm{e}^{-\widehat{\eta}})$,
$(\mathrm{e}^{-\widehat{\eta}},\mathrm{e}^{-\widehat{\lambda}})$ and
$(\mathrm{e}^{-\widehat{\lambda}},1]$. A depiction of the case when not all $\widehat m_i$ are equal can be seen in the first subfigure of Figure \ref{a22}.

    \item \emph{strictly inhomogeneous case:} When the reciprocal of the contraction ratio ($L$) does not divide the number of maps ($M$), and the deterministic projection is an interval e.g. $45$-degree projection of the random Sierpi\'nski carpet. In this case the middle, $(\mathrm{e}^{-\widehat{\eta}},\mathrm{e}^{-\widehat{\lambda}})$ is a proper interval. The function is concave, continuous. Generally it is not known whether it is differentiable in the middle ($\mathrm{e}^{-\widehat{\eta}}\leq p\leq \mathrm{e}^{-\widehat{\lambda}}$)  interval. We remark that the middle interval may be a single point if the set is spacially homogeneous, if reciprocal of the contraction ratio ($L$) does divide the number of maps ($M$), for example in the case of the Mandelbrot percolation.
    \item \emph{The homogeneous case:} when every column sum of every $\mathbf{B}_\theta$ equals
the same constant $K$, so that $M=K\cdot L$. This happens for example for any rational projection
of Mandelbrot percolation (only for rational directions do we have the matrices). In this case the formula is as follows: for $p\leq L/M$ it is $\log(Mp)/\log(L)$, and for $p>L/M$ the value is one. The function is concave, continuous, non-differentiable at $L/M$, analytic otherwise.
\end{itemize}

\begin{figure}
    \centering
    \includegraphics[width=\linewidth]{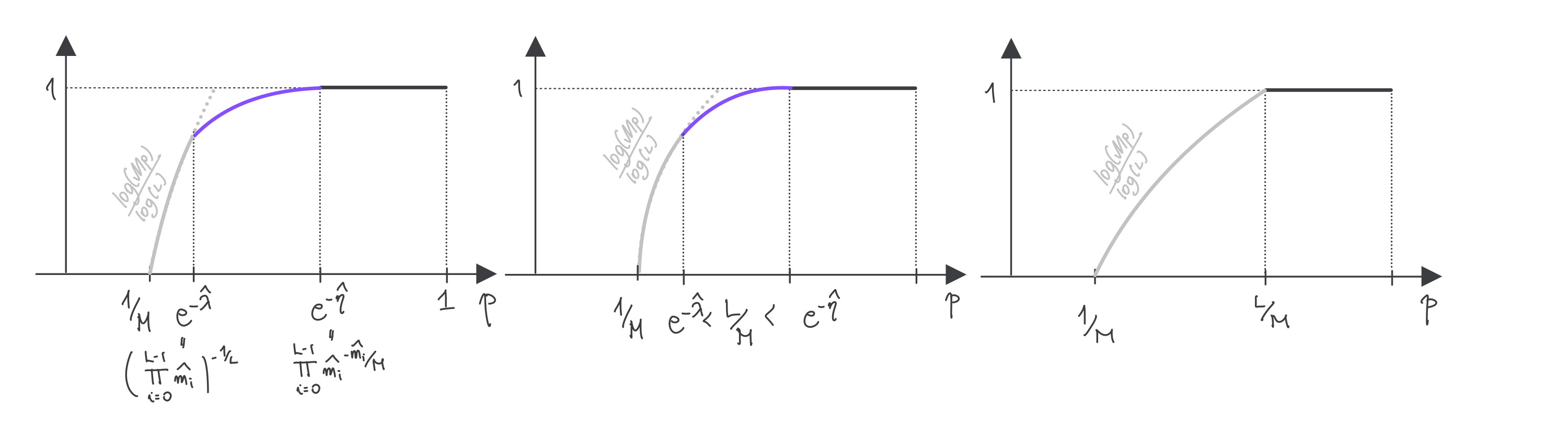}
    \caption{The sketch of the dimension function in case of when the matrices are numbers, in the strictly inhomogeneous case, and in the homogeneous case respectively.}\label{a22}
\end{figure}
For coin-tossing systems, $\lambda=0$ implies that the Lebesgue measure is zero.
This is made precise in the following lemma.
In fact, it is known that this can only happen for a single value of the parameter $p$. This is because for coin-tossing systems the Lyapunov exponent is $\lambda=\log p+ \widehat{\lambda}$, (recall that $\widehat{\lambda}$ is the Lyapunov exponent corresponding to the deterministic system), which is strictly increasing in $p$.

\begin{lemma}
    If the Lebesgue measure of the random attractor is almost surely zero, then $\lambda\leq 0$.
    If $\lambda<0$, the dimension is strictly smaller than the ambient dimension.
    For coin-tossing systems, $\Leb(\Lambda)=0$ holds if and only if $\lambda\leq 0$.
\end{lemma}
\begin{proof}
    The first part follows from Theorem~\ref{measureofattr_theorem} with $\nu$ being the uniform measure, in this case  $\widetilde{\nu}$ is the constant multiple of the Lebesgue measure.
    For the second part, if $\lambda<0$, the infimum of $P(t)$ cannot be attained at $t=0$ since $\lim_{t\to 0+}P'(t)=\lambda<0$ (by Lemma \ref{x46}).
    Therefore, the infimum is attained at some $t\in(0, 1]$, which implies that the dimension is strictly smaller than the ambient dimension.

    The final assertion concerning coin-tossing systems follows from Remark~\ref{rem:cointossing-regular} combined with Theorem~\ref{measureofattr_theorem}.
\end{proof}

\section{Proof of Theorem~\ref{thm:dim}}
We start this section with two preliminary statements, which we will use in the proofs.

\begin{lemma}[{\cite[Theorem 3.5]{zbMATH06808299}}]\label{t99}
    $\dimH (\Lambda)=\dimB (\Lambda)=\dimP(\Lambda)$ almost surely.
\end{lemma}
Lemma \ref{t99} was proven by Sascha Troscheit for more general (graph-directed) systems.
\begin{lemma}\label{a31}
    $\dimH(\Lambda)$ is constant almost surely conditioned on non-extinction.
\end{lemma}
\begin{proof}
    This follows from the fact that having zero $s$-dimensional Hausdorff measure for any $s$ is an inherited property (see Section~\ref{a43}, in particular Lemma \ref{inherited}).
\end{proof}

We prove the upper and lower bounds separately.
The proof has the same structure and follows similar ideas as can be found in \cite{zbMATH04115622, zbMATH04022321}.
The proof of the upper bound is essentially a combination of the proofs for similar statements in those two papers, with minor modifications to fit the matrix setting. We note that the proof in similar form appeared in \cite{menger}, for the convenience of the reader we repeat the argument.

In \cite{zbMATH04115622, zbMATH04022321}, three possible cases are identified for the proof of the lower bound.
(These cases correspond to the phases defined at the end of Section~\ref{a46} but we end up having 4 cases.)

 We will follow the structure of the proofs appearing in \cite{zbMATH04115622,zbMATH04022321}. Similarly to those we consider analogous cases,  however in the matrix setting we must deviate from the original proofs, because generally the matrix norm we are using is not multiplicative.

\subsection{Proof of the upper bound}
For the upper bound, we count the number of level-$n$ $L$-adic cubes required to cover the retained level-$n$ cylinders.

Let $\mathbf{Y}=(Y_0, Y_1, \dots)$ be a sequence of random variables counting the level-$n$ $L$-adic cubes intersecting $\Lambda_n$, which is the $n$-th level approximation of the random set.
Namely,
\begin{equation*}
    Y_0\coloneqq 1, \quad
    Y_n \coloneqq\#\{u\underline{a} \in [N]_1\times[L^d]_0^n: J^{(u)}_{\underline{a}}\cap \Lambda_n \neq \emptyset\}.
\end{equation*}
With this notation, 
\begin{equation}\label{b91}
    \overline{\dim}_\text{B}(\Lambda)\leq\limsup_{n\to \infty}\frac{\log(Y_n)}{n\log (L)}.
\end{equation}

The proof has two main structural parts: 
\begin{equation}\label{a44}
    \limsup_{n\to \infty}\frac{\log (Y_n)}{n} \leq \sup_{n}\frac{\log(\mathbb{E}[Y_n])}{n}=\lim_{n\to\infty} \frac{\log(\mathbb{E}[Y_n])}{n},
\end{equation}
where the inequality holds almost surely, and
\begin{equation}\label{a45}
    \lim_{n\to\infty} \frac{\log(\mathbb{E}[Y_n])}{n}\leq \inf_{t\in[0, 1]}P(t).
\end{equation}

We will see that the equality of \eqref{a44} is a consequence of Fekete's superadditive lemma.
The inequality of \eqref{a44} will follow using Markov's inequality and the Borel--Cantelli Lemma. 

\subsubsection{Proof of equation \eqref{a44}}
Let us recall Fekete's lemma on the limit of a superadditive sequence.
\begin{lemma}[Fekete's lemma]
    If $(b_n)_{n=1}^\infty$ is a superadditive sequence (i.e.\ $b_{n+m}\geq b_n+b_m$ for all $n,m\in\N$), then
    \begin{equation}
        \lim_{n\to \infty}\frac{b_n}{n}\text{ exists, and }\lim_{n\to \infty}\frac{b_n}{n}=\sup_{n}\frac{b_n}{n}.
    \end{equation}
\end{lemma}
Taking logarithms and applying Fekete's lemma to the sequence $b_n=\log\left(a_{n}\right)-\log(k)$, we obtain the following corollary.
\begin{corollary}\label{b96}
    Consider a strictly positive sequence $(a_n)_{n=1}^\infty$.
    Assume that there exists a $k$ such that for all $n$ and $m$, $a_{n+m}\geq \frac{a_n\cdot a_m}{k}$.
    Then the limit $\lim\limits_{n\to \infty}\frac{1}{n}\log (a_n)$ exists and $\lim\limits_{n\to \infty}\frac{1}{n}\log (a_n)=\sup_{n}\frac{1}{n}\log (a_n)$.
\end{corollary}
\begin{definition}[Quasi-supermultiplicative sequence]
    If $(a_n)_{n=1}^\infty$ is a sequence as in Corollary~\ref{b96}, then we call $(a_n)_{n=1}^\infty$ \emph{quasi-supermultiplicative}.
\end{definition}
Since we conditioned on non-extinction of $\#\mathcal{E}_n$, $Y_n>0$ almost surely, and $\mathbb{E}(Y_n)>0$ for each $n$. 
\begin{lemma}
    There exists a $\gamma$ such that
    \begin{equation*}
        \lim_{n\to \infty}\frac{1}{n}\log(\mathbb{E}[Y_n])\eqqcolon\gamma \quad \text{ and } \quad
        \gamma= \sup_{n}\frac{1}{n}\log(\mathbb{E}[Y_n]).
    \end{equation*}
\end{lemma}
\begin{proof}
    In light of Corollary~\ref{b96}, it suffices to prove that the sequence $\mathbb{E}(Y_n)$ is quasi-supermultiplicative: namely that for every $n,m>0$
    \begin{align*}
        \mathbb{E}(Y_{n+m})\gtrsim_{\mathcal{S},d} 
        \mathbb{E}(Y_n)\mathbb{E}(Y_m).
    \end{align*}  
    
        Assume that we are at level $n$.
    The $n$-th approximation of the set intersects $Y_n$ $L$-adic cubes;
    denote them by $K_{1},\dots,K_{Y_n}$. Recall that $\mathcal{E}_n$ is the
    index set of retained level-$n$ maps, and in general $Y_n\neq\#\mathcal{E}_n$.
    Each $S_{\underline{i}}(I)$ is a cube of side $s_n \approx_{\mathcal{S}} L^{-(n-1)}$
    aligned with $\mathcal{D}_{n-1}$. At most constant number of cubes
    $K_1,\dots,K_{Y_n}$ are within distance $s_n$ of $S_{\underline{i}}(I)$. This constant only depends on the IFS $\mathcal{S}$ and the dimension $d$.
    We choose a subset $\widehat{\mathcal{E}}_n\subset\mathcal{E}_n$. We go through the index set $\{1, \dots, Y_n\}$ removing elements in each step starting with $1$ and for the elements of this index set we add elements of $\mathcal{E}_n$. For the first remaining index $j$ we pick $\underline{i}\in\mathcal{E}_n$ with $K_1\cap S_{\underline{i}}(I)\neq\emptyset$,
    add $\underline{i}$ to $\widehat{\mathcal{E}}_n$, and remove the indices $k$ for all cubes $K_k\neq K_j$ within
    distance $s_n$ of $S_{\underline{i}}(I)$. We repeat this with the remaining indices. All cylinder meeting a cube that we didn't discard is disjoint from our cylinder $S_{\underline{i}}(I)$, and each step removes only at most constant number of cubes. We continue selecting the maps for the subset of the index set $\{1, \dots, Y_n\}$ for those $j$ which we didn't remove.
    Therefore, the cylinders $S_{\underline{i}}(I)$, $\underline{i}\in\widehat{\mathcal{E}}_n$,
    are pairwise disjoint and $\#\widehat{\mathcal{E}}_n \gtrsim_{\mathcal{S},d} Y_n$.    
    Stepping forward to level $n+m$, for each $\underline{i}\in\widehat{\mathcal{E}}_n$, we obtain jointly independent processes $\{Y^{\underline{i}}_k\}_{k\in \N}$ (see Section \ref{a98}).
    Each $Y^{\underline{i}}_k$ has the same distribution as the original process $\{Y_k\}_k$.
    Since $\widehat{\mathcal{E}}_n$ contains disjoint level-$n$ cylinders, for any $m \geq 0$
    \begin{equation*}
        Y_{n+m}\overset{d}{\geq} \sum_{\underline{i}\in \widehat{\mathcal{E}}_n}Y_m^{\underline{i}}, 
    \end{equation*}
    where $Y_{m}^{i}$ are random variables that are independent of each other and of $Y_n$ and distributed according to $Y_m$.
    Since $\#\widehat{\mathcal{E}}_n \gtrsim_{\mathcal{S},d} Y_n$
    \begin{equation*}
        \mathbb{E}(Y_{n+m})\geq \mathbb{E}(\#\widehat{\mathcal{E}}_n) \mathbb{E}(Y_m)\gtrsim_{\mathcal{S},d} \mathbb{E}(Y_n) \mathbb{E}(Y_m).
    \end{equation*}
    This finishes the proof of the quasi-supermultiplicativity of the sequence $\mathbb{E}(Y_{n})$.
\end{proof}
\begin{lemma}\label{b89}
    Almost surely, conditioned on non-extinction
    \begin{equation*}
        \limsup_{n\to \infty}\frac{1}{n}\log (Y_n) \leq \gamma.
    \end{equation*}
\end{lemma}
The proof of Lemma~\ref{b89} is the first part of the argument of \cite[Theorem 1]{zbMATH04022321}; we include it for completeness.
\begin{proof}
    From $\gamma=\sup_{n}\frac{1}{n}\log(\mathbb{E}[Y_n])$, it follows that $\log(\mathbb{E}[Y_n])\leq n\gamma$ for all $n$.
    Let $\varepsilon>0$. By Markov's inequality, $\mathbb{P}(\log (Y_n)\geq n (\gamma+\varepsilon))\leq \frac{\mathbb{E}(Y_n)}{\e{n(\gamma+\varepsilon)}} \leq \e{-n\varepsilon}$. Since $\e{-n\varepsilon}$ is summable for any $\varepsilon>0$, it follows from the Borel--Cantelli lemma that
    \begin{equation*}
        \mathbb{P}\left(\limsup_{n\to \infty}\frac{1}{n}\log (Y_n) \leq \gamma\right)=1
    \end{equation*}
    as claimed.
\end{proof}

\subsubsection{Second part: Equation \eqref{a45}}
We now turn to the part where we only consider the expectation $\mathbb{E}(Y_n)$.
The proof in this section essentially follows from~\cite{zbMATH04115622}, with minor adaptations to the matrix setting.
\begin{lemma}\label{b90}
    The following holds:
    \begin{equation*}
        \limsup_{n} \frac{\log(\mathbb{E}[Y_n])}{n} \leq \inf_{t\in(0, 1)}P(t)=\inf_{t\in[0, 1]}P(t).
    \end{equation*}
\end{lemma}
\begin{proof}
    The final equality follows by continuity of $P(t)$.
    It remains to prove the first inequality.
    We will prove that for all $\varepsilon>0$ there exists an $N_0$ such that for all $n>N_0$,
    \begin{equation}\label{a47}
        \frac{\log(\mathbb{E}[Y_n])}{n}\leq \inf_{t\in (0, 1)} P(t)+\varepsilon. 
    \end{equation}
    From \eqref{a47}, the assertion follows.

For a fixed $n\in \N$, $\underline{\theta}\in[L^d]_0^n$
and $u\in [N]_1$ we consider the random variables 
\begin{equation}
    Z_{n}^{(u)}(\underline{\theta})(v)=\#\{\underline{i}\in \mathcal{E}_n: \; S_{\underline{i}}(J^{(v)})=J^{(u)}_{\underline{\theta}}\}.
\end{equation}

  It denotes the number of type-$v$ cubes contained in the $L^{-n}$ cube coded by $\underline{\theta}$ inside the basic cube $J^{(u)}$. 
    Fix $0< t< 1$ Then $f(x)=x^t$ is a concave function.
    Then the generalized Markov's inequality (with the function $f(x)$), Jensen's inequality, and finally the concavity of $f(x)$ yield
    \begin{align*}
        \mathbb{E}(Y_n)  ={}&\mathbb{E}\Big(\sum_{u=1}^{N}\sum_{\underline{\theta}\in[L^d]_0^n}\ind{\sum_{v=1}^{N}Z_n^{(u)}(\underline{\theta})(v)\geq 1}\Big)\leq \sum_{\underline{\theta}\in[L^d]_0^n}\sum_{u,v=1}^{N}\mathbb{P}(Z_n^{(u)}(\underline{\theta})(v)\geq 1)\\
        \leq{}&
        \sum_{\underline{\theta}\in[L^d]_0^n}\sum_{u,v=1}^{N} \mathbb{E}(Z_n^{(u)}(\underline{\theta})(v)^t)\leq \sum_{\underline{\theta}\in[L^d]_0^n}\sum_{u,v=1}^{N} \mathbb{E}(Z_n^{(u)}(\underline{\theta})(v))^t \\ = {} &
        \sum_{\underline{\theta}\in[L^d]_0^n}\sum_{u,v=1}^{N} M_{\underline{\theta}}(u,v)^t \leq 
        \frac{N^2}{N^{2t}} \sum_{\underline{\theta}\in[L^d]_0^n} \|\mathbf{M}_{\underline{\theta}}\|^t.
    \end{align*}
    This completes the proof.
\end{proof}

\subsection{Proof of the lower bound}\label{sec:lowerbound}
We now proceed with the lower bound: namely, under the joint primitivity condition we have
\begin{equation*}
    \dimH(\Lambda)\geq \inf_{t\in [0, 1]}\frac{P(t)}{\log(L)}.
\end{equation*}
We consider four separate cases: when $\lambda>0$ and when $\lambda\leq 0$ and $\eta>0$, $\eta=0$, and $\eta<0$. The proof differs in each of these cases.

  \textbf{Phase~\ref{ph1}}: $\lambda>0$. When the system is in Phase~\ref{ph1}, the infimum of the pressure function $P(t)$ in the interval $[0,1]$ is attained at $0$. Since $P(0)=\log(L^d)$, the value of the formula at $t=0$ is $d$. To verify the statement we need to show that the almost sure dimension of the random attractor is $d$. This is the case, because by the definition of Phase~\ref{ph1}, the Lyapunov exponent $\lambda$ is positive. By Theorem~\ref{measureofattr_theorem}, this implies that the Lebesgue measure of the random attractor is positive (almost surely conditioned on non-extinction), so the assertion holds. We remark that Theorem~\ref{measureofattr_theorem} is about a measure which is a positive multiple of Lebesgue measure, however this implies the positivity of the Lebesgue measure as well.

 \textbf{Phase~\ref{ph2}, \ref{ph4}}: $\lambda\leq0$ and importantly $\eta>0$. In this situation, the infimum of $P(t)$ is attained in the interior of $[0, 1]$. By continuity of $P'(t)$ (see Lemma~\ref{b88}), for each $\varepsilon>0$ we can choose $t^*\in(0, 1)$ such that $0<P'(t^*)<\varepsilon$. For this choice of $t^*$, there exists a unique Gibbs measure $\mu_{t^*}$ as in Lemma~\ref{gibbsm}. The Lyapunov exponent $\lambda(\mu_{t^*})$ corresponding to this ergodic measure is strictly positive by Lemma~\ref{gibbsm} \eqref{a24} since we chose $t^*$ such that $0<P'(t^*)$.
    Hence, using the statement and notation from Theorem~\ref{measureofattr_theorem}, $\widetilde{\mu}_{t^*}(\Lambda)>0$ almost surely conditioned on non-extinction.

    Recall that $\widetilde{\mu}_{t^*}$ is the sum of the measures projected to the basic cubes $J^{(i)}$, $i=1,\dots, N$ by the $L$-adic coding maps, and therefore the Hausdorff dimension of the measure is unchanged upon projection.
        It follows that 
       \begin{align*}
    \dimH(\Lambda)\geq \dimH(\widetilde{\mu}_{t^*})\geq &\dimH(\mu_{t^*})\geq \frac{-{t^*} P'({t^*})+P({t^*})}{\log(L)}\\\geq \frac{-{t^*}\varepsilon+P({t^*})}{\log(L)}
&\geq \frac{-\varepsilon +P({t^*})}{\log(L)}\geq -\frac{\varepsilon}{\log(L)} +\inf_{{t}\in[0,1]}\frac{P({t})}{\log(L)}.
\end{align*}
Letting $\varepsilon\to 0$ completes the proof in Phases~\ref{ph2} and~\ref{ph4}.

   \textbf{Phase~\ref{ph3} and \ref{ph5}}: $\eta=0$ and $\lambda\leq 0$. The main technique is to perturb the model, which modifies its phase  to \ref{ph6}. The proof for Phase \ref{ph6} is quite intricate, hence it is deferred to Section~\ref{b87}. 
        We perturb the probability vector $(p_h)_{h\subset[M]_1}$ and denote it by $\widetilde{\underline{p}}=(\widetilde{p}_h)_{h\in \widetilde{H}}$.
        If $p_{\emptyset}>0$, then $\widetilde{H}=H$; otherwise $\widetilde{H}=H\cup \{\emptyset\}$.
        The original system does not die out if $\sum_{h\in H}p_h\#h>1$, so we may choose $\varepsilon$ such that $(1-\varepsilon)\sum_{h\in H}p_h\#h>1$.

        We perturb the probability vector in the following way:
        for $h\in \widetilde{H}\setminus \{\emptyset\}$: $\widetilde{p}_h=p_h(1-\varepsilon)$ and $\widetilde{p}_{\emptyset}=\varepsilon+ p_{\emptyset}(1-\varepsilon)$.
        In this way $\widetilde{\underline{p}}$ is a probability vector and the expectation matrices corresponding to this distribution satisfy $\|\widetilde{\mathbf{M}}_{\underline{\theta}}\|=(1-\varepsilon)^n\|\mathbf{M}_{\underline{\theta}}\|$, for all $\underline{\theta}\in [L^d]_0^n$.
        Hence, the corresponding pressure is $\widetilde{P}(t)=P(t)+t\log(1-\varepsilon)$. 

        Hence its derivative satisfies
        \begin{equation*}
            \widetilde{P}'(t)=P'(t)+\log(1-\varepsilon),
        \end{equation*}
        which implies that the perturbed model is in Phase~\ref{ph6}.
        Then the infimum over the interval $[0, 1]$ is attained at $1$, where $\widetilde{P}(1)=P(1)+\log(1-\varepsilon)$. 

        We can define a simple coupling between the probability space of the perturbed process and the original as follows.
        First, define the random variable $U$ which is uniform on $[0, 1)$.
        We fix an ordering on $H=\{h_1, \dots, h_{\ell}\}$.
        Define the intervals $I_{h_j}=[\sum_{i<j}p_{h_i}, \sum_{i<j}p_{h_i}+p_{h_j})$.
        There is exactly one $j$ such that $U\in I_{h_j}$.
        Let $X_o$ and $X_p$ be random variables, defined as functions of $U$.
        For $U\in I_{h_j}$ we set $X_o=h_j$.
        For the perturbed variable we set $X_p=h_j$ if $\frac{U-\sum_{i<j}p_{h_i}}{p_{h_j}} \leq 1-\varepsilon$ and otherwise $X_p=\emptyset$.
        In this way $X_o$ and $X_p$ are the coupled versions of $\underline{p}$ and $\widetilde{\underline{p}}$, respectively.
        This is depicted in Figure~\ref{b86}.
        \begin{figure}
            \centering
            \includegraphics[width=0.5 \textwidth]{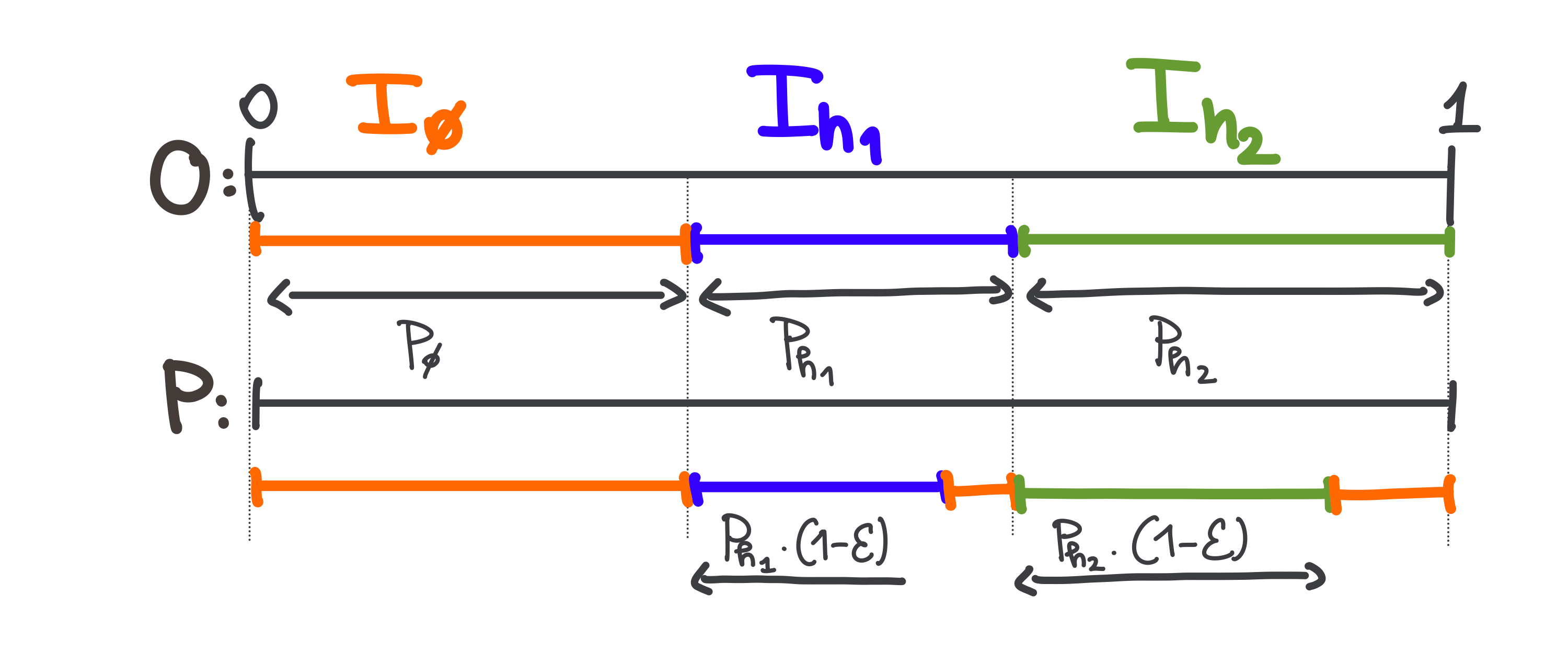}\caption{Depiction of the coupling between the original and the perturbed distribution.
            The value of $X_o$ depends only on which interval $U$ is in, this is shown in the interval denoted by ``O''.
        The value of $X_p$ further depends on whether $U$ is in an orange interval on the part denoted by ``P'', in which case its value is $\emptyset$.}\label{b86}

        \end{figure}

        We construct the labelled Galton–Watson trees $\mathcal{T}_o$ and $\mathcal{T}_p$ with offspring distribution given by $X_o$ and $X_p$, respectively.
        Note that with positive probability the perturbed tree $\mathcal{T}_p$ survives since $\sum_{h\in H}\widetilde{p}_{h}\#h>1$.
        Clearly $\mathcal{T}_p(\omega)\subseteq \mathcal{T}_o(\omega)$ for all $\omega$, and therefore $\dimH (\pi(\partial\mathcal{T}_p(\omega)))\leq \dimH(\pi(\partial\mathcal{T}_o(\omega)))$ for all realizations.

        In this way, assuming that the dimension formula is proved for Phase~\ref{ph6}, it follows that
        \begin{align*}
            \dimH (\pi(\partial\mathcal{T}_p(\omega)))=\inf_{t\in [0, 1]} \frac{\widetilde{P}(t)}{\log(L)}&=\inf_{t\in[0, 1]}\frac{P(t)+t\log(1-\varepsilon)}{\log(L)}\\
                                                                                        &\geq \inf_{t\in[0, 1]}\frac{P(t)+\log(1-\varepsilon)}{\log (L)}.
        \end{align*}
        Letting $\varepsilon\to0$ completes the proof in Phases~\ref{ph3} and~\ref{ph5}.

\subsubsection{Proof of the lower bound, when the model is in \textbf{Phase~\ref{ph6}}}\label{b87}

We start by sketching the steps of the proof.

The fundamental idea hidden behind the technical proof is standard: finding a subset with the same dimension as the original set, but with zero-dimensional slices ensures that the projection does not decrease the dimension.

First, we subdivide the level-$n$ $L$-adic cubes intersecting the level-$n$ approximation into ``light''  and ``heavy'' parts.
The light parts are the level-$n$ $L$-adic cubes which in expectation do not intersect many cylinders.
Lemma~\ref{bigslices} states that the expected number of cylinders intersecting the light part grows at the same exponential rate as the expected number of all cylinders.
Using a large deviation theorem, we show that each of the level-$n$ $L$-adic cubes in the light part is actually light, i.e. does not intersect an exponential number of cylinders (see Lemma~\ref{large_dev_lemma}).
Combining the fact that the number of retained cylinders intersecting the light part grows at the same exponential rate as the total number of cylinders with the observation that each individual cube in the light part intersects only a small number of cylinders, we obtain the desired assertion.

The approach is inspired by \cite{Simon_Rams_2014} where, as an auxiliary lemma, the authors prove in the setting of homogeneous Mandelbrot percolation that if the similarity dimension is less than one, then along all slices the number of cylinders intersecting the slices grows at most at a polynomial rate.
The idea of looking at a subprocess to calculate the dimension is inspired by \cite{zbMATH06830081}.

Since in Phase~\ref{ph6} $P'(1) < 0$, the infimum of $P(t)$ on the interval $[0, 1]$ is attained at $t=1$.
Also note that $P(1) = \log\left( \sum_{h \subset [M]_1} p_h \cdot \#h \right)$.
We will use the following notational shorthands.
\begin{equation*}
    s_0\coloneqq\sum_{h\in H}p_h\cdot \#h>1,\quad  \text{ and } \quad 
    s\coloneqq\frac{\log(s_0)}{\log(L)}.   
\end{equation*}

We will prove that for all $\varepsilon>0$ 
\begin{equation}\label{result}
    \dimH(\Lambda)\geq \frac{\log(s_0)}{\log(L)}-\varepsilon.
\end{equation}

We begin by defining the light (small) parts of $[L^d]_0^n$:
\begin{equation*}
    \mathfrak{S}_n\coloneqq\{\underline{\theta}\in[L^d]_0^n: \: \|\mathbf{M}_{\underline{\theta}}\|<1\}.
\end{equation*}
The $n$-th approximation of the pressure function is 
\begin{equation*}
    P_n(t)\coloneqq\frac{1}{n}\log(\sum_{\underline{\theta} \in [L^d]_0^n} \|\mathbf{M}_{\underline{\theta}}\|^t).
\end{equation*}

\begin{lemma}\label{bigslices}
    There exist $\delta_0\in(0, 1)$ and $N_0$ such that for all $n\geq N_0$,
    \begin{equation*}
        \frac{1}{N}\sum_{\underline{\theta}\in \mathfrak{S}_n}\|\mathbf{M}_{\underline{\theta}}\|\geq s_0^n(1-\delta_0^n/N).
    \end{equation*}
    
\end{lemma}

\begin{proof}
    In this proof we use the notation 
    \begin{equation}
    B_n\coloneqq[L^d]_0^n \setminus \mathfrak{S}_n=\{\underline{\theta}\in[L^d]_0^n: \: \|\mathbf{M}_{\underline{\theta}}\|\geq 1\}
    \end{equation}
 We begin with three preliminary observations: 
     \begin{enumerate}
        \item There exist $\overline{t}>1,\xi>0$, and $N_0$ such that $P(1)-\xi>P_n(\overline{t})$ for all $n\geq N_0$. \label{b85}
        \item $P(1)=\log(\sum_{h\subset [M]} p_h \cdot \#h)+\lim_{n\to \infty}\frac{\log (N) }{n}=\log(s_0)$.
        \item $\sum_{\underline{\theta}\in B_n}\|\mathbf{M}_{\underline{\theta}}\|< (s_0e^{-\xi })^n$ for all $n\geq N_0$.
    \end{enumerate}
The first part follows from a simple calculation using that $\eta=P'(1)<0$. The second part follows from the definition of $M_{\underline{\theta}}(u,v)$. Namely, $\sum_{\underline{\theta}\in [L^d]_0^n}\|\mathbf{M}_{\underline{\theta}}\|$ is simply $N$ times the expected number of retained cylinders at level $n$. From this the expression for $P(1)$ follows. For the third part, observe that for every $n\geq N_0$:
    \begin{equation*}
        \frac{1}{n}\log\left(\sum_{\underline{\theta}\in B_n}\|\mathbf{M}_{\underline{\theta}}\|\right)\leq \frac{1}{n}\log(\sum_{\underline{\theta}\in B_n}\|\mathbf{M}_{\underline{\theta}}\|^{\overline{t}})\leq P_n(\overline{t})< P(1)-\xi.
    \end{equation*}
    The first inequality follows from the definition of $B_n$, namely that $\|\mathbf{M}_{\underline{\theta}}\|\geq 1$ for all $\underline{\theta}\in B_n$. The second follows from $B_n\subset [L^d]_0^n$ and the third from the first observation. Therefore, by the second observation 
    \begin{equation*}
        \sum_{\underline{\theta}\in B_n}\|\mathbf{M}_{\underline{\theta}}\|< \left(s_0 \cdot e^{-\xi}\right)^{n},
    \end{equation*}
    for all $n\geq N_0$.

    For $\xi$ satisfying the assumptions in observation \ref{b85} we fix $\delta_0\coloneqq e^{-\xi}\in(0, 1)$. It follows that
    \begin{equation*}
        \frac{1}{N}\sum_{\underline{\theta}\in \mathfrak{S}_n}\|\mathbf{M}_{\underline{\theta}}\|=\frac{1}{N}\left(\sum_{\underline{\theta}\in [L^d]_0^n}\|\mathbf{M}_{\underline{\theta}}\|-\sum_{\underline{\theta}\in B_n}\|\mathbf{M}_{\underline{\theta}}\|\right) \geq s_0^n \left(1-\delta_0^n/N\right).
    \end{equation*}
    \end{proof}

\begin{definition}[Modified trees]\label{b83}
    \begin{enumerate}
        \item Fix a (deterministic) set $D\subset [M]_1$. We consider the \emph{restricted subtree} $\mathcal{T}_D$ of the random tree $\mathcal{T}$. For each $\omega\in \Omega$ and realization $\mathcal{T}(\omega)$ let $\mathcal{T}_D(\omega)$ be the subtree with root $\emptyset$ and vertices from $D^*$. Namely, for all $n\in \N$ and $\underline{i}\in\mathcal{E}_n$, $\underline{i}\in (\mathcal{T}_D)_n$ if and only if $\underline{i}\in D^n$. The edge set is unchanged between retained vertices. 
        \item The $K$\emph{-step tree} $\mathcal{T}^{(K)}$ is constructed by sampling the original tree $\mathcal{T}$ at cubes of level $nK$. The $n$-th level of this tree, denoted by $\mathcal{T}^{(K)}_n$, is given by:
\begin{equation}
\mathcal{T}^{(K)}_n \coloneqq \mathcal{T}_{nK}.
\end{equation}
The edge set is given by 
\begin{equation}
E^{(K)} = \{ (\underline{i}, \underline{i}\underline{j}) : \underline{i} \in \mathcal{T}^{(K)}_n, \underline{i}\underline{j}\in \mathcal{T}^{(K)}_{n+1} \}.
\end{equation}
    \end{enumerate}

\end{definition}

We can restrict the distribution $\underline{p}$ to deterministic subsets in the following way: fix $D\subset [M]_1$, and we define $\underline{p}_D=(\widetilde{p}_h)_{h\subset [M]_1}$ where
\begin{equation}\label{b84}
    \widetilde{p}_h=\sum_{k\subset [M]_1}p_k\ind_{\{k\cap D=h\}}(k).
\end{equation}
For a distribution $\underline{p}$, let $e(\underline{p})=\sum_{h\subset [M]_1}p_h \#h$.

\begin{lemma}[Basic facts about labelled Galton–Watson trees]\label{tree-facts}
    Assume $\mathcal{T}$ is a labelled Galton–Watson tree with offspring distribution $\underline{p}=(p_h)_{h\subset[M]_1}$.
    Then the following hold: 
    \begin{enumerate}
        \item Fix a deterministic set $D\subset [M]_1$. The restricted tree $\mathcal{T}_D$ is a labelled Galton–Watson tree with offspring distribution $\underline{p}_D=(\widetilde{p}_h)_{h\subset [M]_1}$, defined in \eqref{b84}. The boundary of this tree is non-empty with positive probability if and only if $e(\underline{p}_D)>1$ or there is exactly one child at each level.
        \item For a fixed $K\in \N$, the $K$-step tree $\mathcal{T}^{(K)}$ is also a branching tree with offspring distribution equal to the law of $\mathcal{E}_K$ on $[M]_1^K$. The boundary is non-empty with positive probability if and only if the boundary of the original tree is non-empty with positive probability.
        \item Assume that $\mathcal{T}$ is equipped with the metric $d(\underline{i}, \underline{j})=\rho^{|\underline{i}\wedge \underline{j}|}$ for some $\rho\in(0, 1)$. The Hausdorff dimension of its boundary $\partial \mathcal{T}$ equals $-\log(e(\underline{p}))/\log(\rho)$ almost surely conditioned on non-extinction.
    \end{enumerate}
    
\end{lemma}
\begin{proof}
    The first two statements follow simply from the analogous properties of the original tree. For the proof of the third statement see \cite{zbMATH03734900}.
\end{proof}
For the rest of the proof, we fix $K$ such that 
\begin{itemize}
    \item $K>N_0$,  
    \item $\log(s_0^K(1-\delta_0^K))/(K\log(L))\geq s-\varepsilon/3$,  and
    \item $s_0^K(1-\delta_0^K)>1$ (which is possible, since $s_0>1$).
\end{itemize}
By Fact~\ref{a96}, for each $\underline{i}\in [M]_1^K$ and each $v\in [N]_1$
there exist unique $u\in [N]_1$ and $\underline{\theta}\in [L^d]_0^K$ such that
$S_{\underline{i}}(J^{(v)})=J^{(u)}_{\underline{\theta}}$;
we denote this $\underline{\theta}$ by $\underline{\theta}(v,\underline{i})$.
We consider the set $\mathfrak{D}_K\subset[M]_1^K$  defined as 
\begin{multline*}
    \mathfrak{D}_K=\{\underline{i}\in [M]_1^K:\; \forall v \in [N]_1\;  \exists \underline{\theta}\in \mathfrak{S}_K,\, u\in [N]_1\; s.t.\;  S_{\underline{i}}(J^{(v)})=J^{(u)}_{\underline{\theta}} \}\\=\{\underline{i}\in [M]_1^K:\; \forall v \in [N]_1\; \underline{\theta}(v,\underline{i})\in \mathfrak{S}_K \}.
\end{multline*}
This set contains those $K$-compositions of the functions of the IFS such that the fibers, in expectation, intersect only a small number of cylinders.

With the notation of Definition~\ref{b83} we consider the tree 
\begin{equation*}
    \widetilde{\mathcal{T}}=(\mathcal{T}^{(K)})_{\mathfrak{D}_K}.
\end{equation*}
Let $\widetilde{\mathcal{E}}_n$ denote the level-$n$ nodes of the tree $\widetilde{\mathcal{T}}$. Further, on the boundary $\partial \widetilde{\mathcal{T}}$, we define the metric $d(\mathbf{i}, \mathbf{j})=L^{-K|\mathbf{i}\wedge \mathbf{j}|}$, where $|\mathbf{i}\wedge \mathbf{j}|$ is the length of the longest common prefix of $\mathbf{i}\neq \mathbf{j}$, and $d(\mathbf{i}, \mathbf{i})=0$.
We collect the properties of $\widetilde{\mathcal{T}}$ in the following corollary of Lemma~\ref{tree-facts}.
\begin{corollary}[Properties of $\widetilde{\mathcal{T}}$]\label{a37}
    The tree $\widetilde{\mathcal{T}}$ is a labelled Galton–Watson tree, satisfying the following properties.
    \begin{enumerate}
    \item Let $\underline{p}^{(K)}=(p^{(K)}_g)_{g\subset [M]_1^K}$ be the distribution of $\mathcal{E}_K$. The offspring distribution of $\widetilde{\mathcal{T}}$ is $\widetilde{\underline{p}}=(\widetilde{p}_h)_{h\subset [M]_1^K}$
    \begin{equation}
        \widetilde{p}_h=\sum_{\substack{g\subset [M]_1^K\\ g\cap \mathfrak{D}_K=h}}p^{(K)}_g
    \end{equation}
    on the set $\mathcal{P}([M]_1^K)$, and 
    \item $e(\widetilde{\underline{p}})\geq s_0^K(1-\delta_0^K)$, hence
    \item with positive probability $\partial \widetilde{\mathcal{T}}\neq \emptyset$ and
    \item with respect to the metric $d$, $\dimH(\partial \widetilde{\mathcal{T}})\geq s-\varepsilon/3$ almost surely conditioned on non-extinction. 
    \end{enumerate}
\end{corollary}

\begin{proof}
    The proof follows from Lemma~\ref{tree-facts}, considering that
    \begin{align*}
         e(\widetilde{\underline{p}})& =
        \sum_{h\subset [M]_1^K}  \#h \widetilde{p}_h
        =\sum_{h\subset \mathfrak{D}_K} \#h \sum_{g\subset [M]_1^K} \mathbb{P}(\mathcal{E}_K=g)\ind\{\mathfrak{D}_K\cap g=h\} \\
        &= \sum_{h\subset \mathfrak{D}_K} \sum_{\underline{i}\in [M]_1^K} \sum_{g\subset [M]_1^K} \mathbb{P}(\mathcal{E}_K=g)\ind\{\underline{i}\in h, \mathfrak{D}_K\cap g=h\}=
        \sum_{h \subset\mathfrak{D}_K}\sum_{\underline{i}\in h}\mathbb{P}(\underline{i}\in \mathcal{E}_K) \\
        &=\sum_{\underline{i}\in [M]_1^K} \ind\{\forall v\in [N]_1:\: \underline{\theta}(v,\underline{i})\in \mathfrak{S}_K\}\cdot \mathbb{P}(\underline{i}\in \mathcal{E}_K) \geq  \sum_{\underline{i}\in[M]_1^{K}}(1-N) \cdot \mathbb{P}(\underline{i}\in \mathcal{E}_K)+\sum_{\underline{\theta}\in \mathfrak{S}_K}\|\mathbf{M}_{\underline{\theta}}\| \\ &\geq  s_0^K(1-\delta_0^K).
    \end{align*}
    Now we explain the last two inequalities.

    The first inequality follows from applying the union bound for the complement.
    \begin{align*}
        &\ind\{\forall v\in [N]_1:\: \underline{\theta}(v,\underline{i})\in \mathfrak{S}_K\} \geq   1- \sum_{v\in[N]_1} \ind\{\underline{\theta}(v,\underline{i})\notin \mathfrak{S}_K\}\\ = & 1- (N-\sum_v \ind\{\underline{\theta}(v,\underline{i})\in \mathfrak{S}_K\}).
    \end{align*}

    For the last inequality we use $\sum_{\underline{i}\in[M]_1^K}\mathbb{P}(\underline{i}\in\mathcal{E}_K)=\mathbb{E}(\#\mathcal{E}_K)=s_0^K$ and Lemma~\ref{bigslices} with $n=K$ (recall that $K> N_0$):
\begin{equation*}
    (1-N)\,s_0^K+\sum_{\underline{\theta}\in \mathfrak{S}_K}\|\mathbf{M}_{\underline{\theta}}\|
    \geq (1-N)\,s_0^K+N s_0^K\left(1-\frac{\delta_0^K}{N}\right)
    = s_0^K\left(1-\delta_0^K\right).
\end{equation*}
\end{proof}

 Recall that for $\underline{\theta}\in [L^d]_0^n$ and $u,v \in \{1, \dots, N\}$, $Z_n^{(u)}(\underline{\theta})(v)=\#\{\underline{i}\in \mathcal{E}_n:\; S_{\underline{i}}(J^{(v)})=J_{\underline{\theta}}^{(u)} \}$. They 
form the vectors $\underline{Z}_{n}^{(u)}(\underline{\theta})=(Z_{n}^{(u)}(\underline{\theta})(v))_{v\in [N]_1}$. For $\underline{\beta}=(\underline{\theta}_1,\dots, \underline{\theta}_n)$, $\underline{\theta}_i\in[L^d]_0^K$ we define 
 $\widetilde{\underline{Z}}_n^{(u)}(\underline{\beta})$ corresponding to $\widetilde{\mathcal{E}}_n$ analogously. We remark that for $\underline{\beta}\in \mathfrak{S}_K^n$ the vector random variables $\underline{Z}_{nK}^{(u)}(\underline{\beta})$ upper bounds $\widetilde{\underline{Z}}_{n}^{(u)}(\underline{\beta})$. In what follows, we fix a realization $\omega\in \Omega$.
Before the final proof we present two auxiliary lemmas, Lemma \ref{a53} and \ref{lem:uniform-bound}. The proof of these is deferred to the end of the section.
\begin{lemma}\label{a53}
Assume that for $\delta>0$, $c>0$ and $N_0\in \N$, for every $m>N_0$, for all $\underline{\beta}\in \mathfrak{S}_K^m$ and $u,v\in[N]_1$
\begin{equation*}
    \widetilde{Z}_{m}^{(u)}(\underline{\beta})(v)(\omega) \leq c \cdot L^{m\delta}.
\end{equation*}
Then 
\begin{equation*}
    \dimH (\partial \widetilde{\mathcal{T}}(\omega))\leq \dimH(\pi(\partial \widetilde{\mathcal{T}}(\omega)))+2\delta,
\end{equation*}
where $\pi$ is the natural projection corresponding to the deterministic IFS $\mathcal{S}$.
\end{lemma}

\begin{lemma}\label{lem:uniform-bound}
Fix $\delta>0$.  There exists an event $\Omega_{\delta}$ with $\mathbb{P}(\Omega_{\delta})>0$, such that for $\omega\in \Omega_{\delta}$:
\begin{enumerate}
  \item $\partial \widetilde{\mathcal{T}}(\omega)\neq \emptyset$.
  \item There exist constants $C(\omega)>0$, and $\widehat{N}_0(\omega)$ such that for $m>\widehat{N}_0$ for every $\underline{\beta} \in \mathfrak{S}_K^m$ and every $u,v\in[N]_1$,
  \begin{equation}\label{eq:uniform-bound}
    \widetilde{Z}_m^{(u)}(\underline{\beta})(v)(\omega) \le C\cdot L^{m\delta}.
  \end{equation}
\end{enumerate}
\end{lemma}

The lower bound in \eqref{result} now follows by combining the previous lemmas. 
Recall that $\pi\colon\Sigma^{[M]_1}\to \R^d$ is the natural projection corresponding to the IFS $\mathcal{S}$.
From $\widetilde{\mathcal{T}}(\omega)\subset \mathcal{T}(\omega)$, it follows that for any $\omega\in\Omega$,
\begin{equation}
    \dimH(\pi(\partial\widetilde{\mathcal{T}}(\omega)))\leq \dimH(\pi(\partial  \mathcal{T}(\omega))).
\end{equation}
Choose $\delta=\varepsilon/3$. Combining Lemma~\ref{lem:uniform-bound} and Lemma~\ref{a53} gives that there exists $\Omega_{\delta}$, with $\mathbb{P}(\Omega_{\delta})>0$, such that for all $\omega\in \Omega_{\delta}$,
\begin{equation*}
    \dimH(\partial \widetilde{\mathcal{T}}(\omega))\leq \dimH(\pi(\partial \widetilde{\mathcal{T}}(\omega)))+2\varepsilon/3.
\end{equation*}
From Corollary~\ref{a37},
\begin{equation*}
    \dimH(\partial\widetilde{\mathcal{T}}(\omega))\geq s-\varepsilon/3.
\end{equation*}
Combining these with Lemma~\ref{a31} gives the desired result regarding the dimension formula in Phase~\ref{ph6}. This concludes the proof of the lower bound since by Lemma \ref{a31} the Hausdorff dimension of $\Lambda=\pi(\partial \mathcal{T})$ is constant almost surely conditioned on non-extinction.

\bigskip

\textbf{Proofs of Lemmas~\ref{a53}-\ref{lem:uniform-bound}}
\smallskip

\begin{proof}[Proof of Lemma~\ref{a53}]
    Let $t\coloneqq\dimH(\pi(\partial \widetilde{\mathcal{T}}(\omega)))$. Since $\mathcal{H}^{t+\delta}(\pi(\partial \widetilde{\mathcal{T}}(\omega)))=0$, for every $m> N_0$ there is a cover $\{I_i\}_{i\in \mathcal{I}}$ of $\pi(\partial \widetilde{\mathcal{T}}(\omega))$ with $I_i\in\mathcal{D}_{K\ell(I_i)-1}$, $\ell(I_i)\geq m$ for all $i\in\mathcal{I}$, and
    \begin{equation*}
        \sum_{i\in \mathcal{I}} |I_i|^{t+\delta}<1.
    \end{equation*}
    We pull this cover back to the symbolic space. For such a cube $I$, let
    \begin{equation*}
        g(I)=\{\underline{j}\in \widetilde{\mathcal{E}}_{\ell(I)}:\ \exists u\in[N]_1,\ S_{\underline{j}}(J^{(u)})\cap I\neq\emptyset\}.
    \end{equation*}
    Then $\bigcup_{i\in \mathcal{I}}\bigcup_{\underline{k}\in g(I_i)}[\underline{k}]$ covers $\partial\widetilde{\mathcal{T}}(\omega)$, and each $[\underline{k}]$ has diameter $L^{-K\ell(I_i)}\le L^{-Km}$.
    For $\underline{j}\in\widetilde{\mathcal{E}}_{\ell(I)}$, each $S_{\underline{j}}(J^{(u)})$ is a cube $J^{(u')}_{\underline\beta}\in\mathcal{D}_{K\ell(I)-1}$ with $\underline\beta\in \mathfrak{S}_K^{\ell(I)}$, and at most $3^d$ cubes of $\mathcal{D}_{K\ell(I)-1}$ meet $I$. Hence, by the assumption of the lemma,
    \begin{equation*}
        \#g(I)\lesssim_{d,N,c}  L^{K\ell(I)\delta}.
    \end{equation*}
    Using $L^{-K\ell(I)}\leq |I|$, the cost of this cover is
    \begin{align*}
        \sum_{i\in \mathcal{I}}\sum_{\underline{k}\in g(I_i)} |[\underline{k}]|^{t+2\delta}
        &=\sum_{i\in \mathcal{I}} \#g(I_i)\, L^{-K\ell(I_i)(t+2\delta)}
        \lesssim_{d,N,c}\sum_{i\in \mathcal{I}} L^{K\ell(I_i)\delta}\, L^{-K\ell(I_i)(t+2 \delta)}\\
        &= \sum_{i\in \mathcal{I}} L^{-K\ell(I_i)(t+\delta)}
        \leq  \sum_{i\in \mathcal{I}} |I_i|^{t+\delta}<1.
    \end{align*}
    Since $m$ was arbitrary, $\mathcal{H}^{t+2\delta}(\partial\widetilde{\mathcal{T}}(\omega))<\infty$, hence $\dimH(\partial\widetilde{\mathcal{T}}(\omega))\le t+2\delta$.
\end{proof}
For the proof of Lemma \ref{lem:uniform-bound} we need an auxiliary lemma. The statement of the lemma relies on a certain comparison function.

For the fixed parameters $\widetilde{\delta}>0$ and $\alpha>0$ we define a comparison function.
Its arguments are $D\in\N$, $n\in\N$, $\underline{\theta}\in [L^d]_0^K$, and $u,v\in [N]_1$, and it is given by
\begin{equation*}
    \phi(D,n,\underline\theta,u,v)\coloneqq
    \begin{cases}
        D\cdot M_{\underline\theta}(u,v)\,(1+\alpha), & D\ge L^{(n-1)\widetilde{\delta}},\\[4pt]
        D \cdot M^K, & D<L^{(n-1)\widetilde{\delta}}.
    \end{cases}
\end{equation*}
Recall that $M_{\underline{\theta}}(u,v)$ is the expected number 
of maps sending $J^{(v)}$ to $J^{(u)}_{\underline{\theta}}$. 
The number of functions is $M$, and $M^K$ is a fixed uniform upper bound for the number of $K$-compositions. Hence, to summarize, the function $\phi$ defines the threshold for typical growth: for large populations, we bound the count by the expectation (with a small error margin), while for small populations, we fall back on the worst-case deterministic bound $D\cdot M^K$.

We will consider the bad event:
\begin{multline*}
F_n\coloneqq\Big\{\exists \underline\beta=(\beta_1,\dots, \beta_n)\in \mathfrak{S}_K^n,\ \exists u, v\in[N]_1:\\
Z_{nK}^{(u)}(\underline\beta)(v) \;>\; \sum_{w\in[N]_1} \phi\big(Z^{(u)}_{(n-1)K}(\underline\beta|_{n-1})(w),n,\,\beta_n,\,w,\,v\big)\Big\}.
\end{multline*}
In what follows, when the arguments of $Z$ are clear from the context we will omit them. 
\begin{lemma}\label{large_dev_lemma}
    For $\widetilde{\delta}>0$ and $\alpha>0$ there exists $\varrho\in(0,1)$ such that for $n\in \N$:
    \begin{equation*}
        \mathbb{P}(F_n)\leq (L^{dK})^n N^3 \varrho^{L^{(n-1)\widetilde{\delta}}}.
    \end{equation*}
\end{lemma}

\begin{proof}[Proof of Lemma~\ref{large_dev_lemma}]

    By the union bound,
\[
\mathbb{P}(F_n) \le \sum_{\underline\beta\in \mathfrak{S}_K^n}\sum_{u,v\in[N]_1}
\mathbb{P}\Big( Z_{nK}^{(u)}(\underline\beta)(v) > \sum_{w} \phi(\cdot)\Big).
\]

\smallskip

Fix $\underline\beta\in \mathfrak{S}_K^n$ and $u,v\in [N]_1$. We can write $Z_{nK}^{(u)}(\underline\beta)(v)$ as a sum of independent contributions coming from members of generation $(n-1)K$:
\[
Z_{nK}^{(u)}(\underline\beta)(v)
\overset{d}{=} \sum_{w\in[N]_1}\sum_{j=1}^{D_{w}} Y_j(\beta_n,w,v),
\]
where $D_{w}\coloneqq Z_{(n-1)K}^{(u)}(\underline\beta)(w)$ and for fixed $\beta_n, w, v$ the random variables $Y_j(\beta_n,w,v)$ are i.i.d.\ with the same distribution as $Z^{(w)}_K(\beta_n)(v)$ (hence with mean $M_{\beta_n}(w,v)$). Here we use the offspring independence and identical distribution assumption at each parent (see Section \ref{a98}).
In what follows, we write $Z_{nK}(\underline\beta)(v)=Z_{nK}^{(u)}(\underline\beta)(v)$, since the dependence on $u$ will not have any effect.
In this way 
\begin{align*}
    \mathbb{P}\Big( Z_{nK}(\underline\beta)(v) > \sum_{w} \phi(\cdot)\Big)&= \mathbb{P}\Big( \sum_{w\in[N]_1}\sum_{j=1}^{D_{w}} Y_j(\beta_n,w,v) > \sum_{w} \phi(\cdot)\Big)\\
                                                                          &\leq \sum_{w\in[N]_1} \mathbb{P}\Big( \sum_{j=1}^{D_{w}} Y_j(\beta_n,w,v)> \phi(\cdot)\Big).
\end{align*}
For a fixed $w$ we use the law of total probability to consider the two separate cases when $D_w$ is big and when it is small, namely:
\begin{align*}
     &\mathbb{P}\Big( \sum_{j=1}^{D_{w}} Y_j(\beta_n,w,v) > \phi(\cdot)\Big)\\
    ={}&\mathbb{P}\Big( \sum_{j=1}^{D_{w}} Y_j(\beta_n,w,v) > \phi(\cdot)|D_{w}< L^{(n-1)\widetilde{\delta}}\Big)\mathbb{P}(D_{w}< L^{(n-1)\widetilde{\delta}}) 
    \\+{}&
    \mathbb{P}\Big( \sum_{j=1}^{D_{w}} Y_j(\beta_n,w,v) > \phi(\cdot)|D_{w}\geq L^{(n-1)\widetilde{\delta}}\Big)\mathbb{P}(D_{w}\geq L^{(n-1)\widetilde{\delta}}).
\end{align*}

The first part is simply $0$, because the number of maps is $M$, hence in $K$ levels the number of children can not grow faster than $M^K$. For the second part we use large deviation theory (see Lemma \ref{large_dev_lem}) for the independent sum. This gives:

\begin{equation*}
    \mathbb{P}\Big( \sum_{j=1}^{D_{w}} Y_j(\beta_n,w,v) > \phi(\cdot)|D_{w}\geq L^{(n-1)\widetilde{\delta}}\Big)\mathbb{P}(D_{w}\geq L^{(n-1)\widetilde{\delta}})<\varrho^{ L^{(n-1)\widetilde{\delta}}}
\end{equation*}
for some $\varrho\in(0,1)$ (depending only on the chosen $\alpha$ and the offspring law).
Because for a fixed $\alpha$ the possible offspring distributions run over only a finite set, we may take a uniform $\varrho\in(0,1)$ with the property above for all triples $(w,v,\underline\beta_n)$.

Adding up everything indeed gives the desired result.
\end{proof}

\begin{proof}[Proof of Lemma~\ref{lem:uniform-bound}]

Fix some $0<\widetilde{\delta}<\delta$. 
Let $\mathfrak{m}=\max_{\underline{\theta}\in \mathfrak{S}_K}\|\mathbf{M}_{\underline{\theta}}\|$. Since $\underline{\theta}\in \mathfrak{S}_K$ the matrices $\mathbf{M}_{\underline{\theta}}$ have norm less than one,
we can choose $\alpha$ small enough such that $(1+\alpha)\mathfrak{m}<1$. 
We use Lemma~\ref{large_dev_lemma} for $\alpha$ and $\widetilde{\delta}$.
This implies that
\begin{equation*}
    \sum_n \mathbb{P}(F_n)\leq N^3 \sum_n (L^{dK})^n \varrho^{L^{(n-1)\widetilde{\delta}}}<\infty.
\end{equation*}
Hence, by the Borel–Cantelli lemma $\mathbb{P}(\bigcap_n \bigcup_{k\geq n}F_k)=0 $.

In particular, there exists a subset $\Omega_{BC}\subset \Omega$ with $\mathbb{P}(\Omega_{BC})=1$, such that for all $\omega\in\Omega_{BC}$, there exists an $\widehat{N}_1\coloneqq \widehat{N}_1(\omega)$ such that for all $n\geq \widehat{N}_1$, for all $\underline{\beta}=(\beta_1,\dots, \beta_n)\in \mathfrak{S}_K^{n}$, and $u,v\in [N]_1$:

\begin{equation*}
    Z_{nK}^{(u)}(\underline{\beta})(v)(\omega)\leq \sum_{w\in[N]_1}\phi(Z_{(n-1)K}^{(u)}(\underline{\beta}|_{n-1})(w)(\omega),n, \beta_{n}, w, v).
\end{equation*}
We choose $\widehat{N}_0$ such that $\widehat{N}_0\geq \widehat{N}_1$, and further, for $n\geq \widehat{N}_0$ we have $L^{(n-1)(\widetilde{\delta}-\delta)}<L^{\delta}-\mathfrak{m}(1+\alpha)$.
We choose $C(\omega)= N\cdot M^{\widehat{N}_0(\omega)\cdot K}$. In this way,  for all $\underline{\beta}\in \mathfrak{S}_K^{\widehat{N}_0}$ and types $u,v\in[N]_1$
\begin{equation*}
    Z_{\widehat{N}_0\cdot K}^{(u)}(\underline{\beta})(v)(\omega)<C L^{\widehat{N}_0\cdot \delta}.
\end{equation*}
Now we prove that the analogous statement \eqref{eq:uniform-bound} holds for all $n$. From the choice of our variables, induction on $n$ shows that
for $n>\widehat{N}_0$ and $\underline{\beta}=(\beta_1, \dots, \beta_n)\in \mathfrak{S}_K^n$
\begin{align*}
    Z_{nK}^{(u)}(\underline{\beta})(v) \;<\;&  \sum_{w\in[N]_1} \phi\big(Z_{(n-1)K}^{(u)}(\underline{\beta}|_{n-1})(w),\,n,\,\beta_n,\,w,\,v\big)\\
\leq {}& N\cdot M^K  L^{(n-1)\widetilde{\delta}}+\sum_{w\in[N]_1} M_{\beta_n}(w,v)\cdot Z_{(n-1)K}^{(u)}(\underline{\beta})(w)(1+\alpha)\\
\leq {}& N\cdot M^K  L^{(n-1)\widetilde{\delta}}+ C L^{(n-1)\delta} \mathfrak{m}\cdot (1+\alpha) \\
\leq{}& C L^{(n-1)\delta}( L^{(n-1)(\widetilde{\delta}-\delta)}+ \mathfrak{m}\cdot (1+\alpha))\leq CL^{n\delta}.
\end{align*}

Since $\Omega_{BC}$ is a full probability set, the probability of the intersection with the set of those realizations where the `trimmed' process survives remains positive. This intersection gives the resulting event $\Omega_{\delta}$ since $\widetilde{Z}_n\leq Z_{nK}$.
\end{proof}

\subsection{Conclusion of the proof of Theorem~\ref{thm:dim}}
By \eqref{b91}, Lemma~\ref{b89} and Lemma~\ref{b90}, almost surely
\begin{equation*}
    \overline{\dim}_B(\Lambda)\leq \inf_{t\in[0,1]}\frac{P(t)}{\log(L)},
\end{equation*}
and by Section~\ref{sec:lowerbound}, almost surely conditioned on non-extinction,
\begin{equation*}
    \dimH(\Lambda)\geq \inf_{t\in[0,1]}\frac{P(t)}{\log(L)}.
\end{equation*}
Since $\dimH\Lambda\le\underline{\dim}_B\Lambda\le\overline{\dim}_B\Lambda$, and $\dimH(\Lambda)=\dimB(\Lambda)=\dimP(\Lambda)$ by Lemma~\ref{t99}, this proves \eqref{a42}.

\appendix

\section{Dimensions of sets}
In this section we formally introduce the Hausdorff measure and the Hausdorff, lower and upper box dimensions.
We follow \cite{zbMATH07759243,zbMATH00739280,falc_geom_fractal_sets}.

For a set $U\subset \R^d$ write $\diam(U)=\sup\{\|x-y\|:\, x, y \in U\}$.
A collection of sets $\{U_i\}_{i\in \mathcal{I}}$ is called a \emph{$\delta$-cover} of a set $E$ if $E\subset \bigcup_{i\in\mathcal{I}} U_i$ and, for all $i\in \mathcal{I}$, $\diam(U_i)\leq \delta$.
For $E\subset \R^d$ and $s\geq 0$ we define
\begin{equation*}
    \mathcal{H}^{s}_{\delta}(E)=\inf\left\{\sum \diam(U_i)^s, \, \{U_i\} \text{ is a countable }\delta\text{-cover of E}\right\}.
\end{equation*}
It is easy to see that $\mathcal{H}^s_{\delta}$ is non-increasing in $\delta$ and therefore the limit 
\begin{equation*}
    \mathcal{H}^{s}(E)=\lim_{\delta \to 0+}\mathcal{H}^{s}_{\delta}(E)
\end{equation*}
exists.
The resulting set-valued function $\mathcal{H}^s$ is in fact a metric outer measure which when restricted to the $\mathcal{H}^s$ measurable subsets of $\R^d$ gives the Hausdorff $s$-measure. 
The \emph{Hausdorff dimension} of $E\subset\R^d$ is
\begin{equation*}
    \dimH E\coloneqq\inf\{s\geq 0:\ \mathcal{H}^s(E)=0\}=\sup\{s\geq0:\ \mathcal{H}^s(E)=\infty\}.
\end{equation*}
For a finite Borel measure $\mu$ we define its Hausdorff dimension as
\begin{equation*}
    \dimH\mu\coloneqq\inf\{\dimH A:\ A \text{ Borel},\ \mu(A)>0\}.
\end{equation*}
The same definitions are used on the symbolic space $\Sigma^{[L^d]_0}$ equipped with the metric of Section~\ref{a46}.

Assume that $E$ is a bounded set in $\R^d$.
We define the \textit{upper} and \textit{lower box dimensions} of $E$ as 
\begin{align}
    \underline{\dim}_B(E)\coloneqq&\liminf_{n\to\infty}\frac{\log\#\{U_i\in \mathcal{D}_n: U_i\cap E\neq \emptyset\}}{n\log (L)}\\
    \overline{\dim}_B(E) \coloneqq  &\limsup_{n\to\infty}\frac{\log\#\{U_i\in \mathcal{D}_n: U_i\cap E\neq \emptyset\}}{n\log (L)}.
\end{align}
When they agree we say that the \emph{box dimension} exists, and we denote the common value by $\dim_B(E)$. For the definition and basic properties of the \emph{packing dimension} $\dim_P$ we refer to \cite[Chapter 5]{zbMATH00739280}.

\section{Symbolic space and trees}\label{app:symbolictrees}
It is convenient for us to work with symbolic spaces and trees.
Therefore, we now introduce some basic notation for use throughout this note.
The symbolic space over the alphabet $\{1, \dots, M\}$ is $\Sigma=\Sigma^{[M]_1}=\{1,\dots, M\}^{\N}$, equipped with the metric
\begin{equation*}
    d(\mathbf{i}, \mathbf{j})=\rho^{|\mathbf{i} \wedge \mathbf{j}|}
\end{equation*}
for some fixed $\rho\in(0,1)$. Here $\mathbf{i} \wedge \mathbf{j}$ is the maximal common prefix of $\mathbf{i}$ and $\mathbf{j}$ and $|.|$ denotes the length.
We use this metric space to encode the points in the attractor in the following way.
Let $\pi\colon\Sigma\to \R^d$ be defined as 
\begin{equation*}
    \pi(\mathbf{i})=\lim_{n\to \infty} S_{\mathbf{i}|_n}(\underline{0}),\,\text{for }\mathbf{i}\in \Sigma,
\end{equation*}
where $\mathbf{i}|_n$ stands for the first $n$ letters of $\mathbf{i}$.
This is called the \textit{natural projection} corresponding to the IFS $\mathcal{S}$.

It will also be convenient to view the space $\Sigma$ as an $M$-ary tree.
We introduce some basic tree terminology, based on the introduction of the book
\cite{zbMATH06653785} (which we also recommend for further graph terminology).



We will always work with locally finite rooted trees, where each parent has only a finite number of children.
In fact, we will always assume that for some $2\leq M\in \N$ each parent has at most $M$ children.
Sometimes we will work with $M$-ary trees, by which we mean that each vertex has exactly $M$ children.
A vertex $v$ is called a level-$n$ vertex if the shortest directed path from the root to $v$ has exactly $n+1$ vertices. 

We represent a rooted tree in which each vertex has at most $M$ children using finite words over the alphabet $[M]_1$.  
Let us describe the construction of the full rooted tree.
The root of the tree is the empty word, denoted by $\emptyset$.
Recall that for $n\geq 0$, we write 
\[
\Sigma_n = \{1,\dots,M\}^n, \qquad \Sigma_* = \bigcup_{n\ge 0} \Sigma_n ,
\]
where we set $\Sigma_0 = \{\emptyset\}$.
In the tree the vertex set is $\Sigma_*$. The edge set is defined by
\[
E = \{\, (\underline{i},\underline{i}i_{n+1}) : 
     \underline{i} \in \Sigma_*,\ i_{n+1} \in [M]_1 \,\}.
\]
Thus, each word or vertex is connected to its children obtained by appending one symbol. Every vertex $\underline{i}\neq\emptyset$ of length $n\geq 1$  has a unique parent 
$\underline{i}|_{n-1}$.

More generally, we may consider subtrees which are obtained as subsets of the full tree constructed above.
For a general tree $\mathcal{T}$, the boundary, $\partial \mathcal{T}$ of $\mathcal{T}$ consists of the infinite words $\mathbf{i}$ for which all finite prefixes $\mathbf{i}|_n$ ($n\in \N$) are vertices of $\mathcal{T}$.
Since we will use trees to represent the geometric structure of a random set obtained by the percolation process defined earlier, finite lineages (that is, vertices in the tree which have no children) will occur naturally.
Such lineages are invisible to the boundary of the tree.

\section{Survival of the process, and the event $\Lambda\neq \emptyset$}\label{survival}
In this section we discuss the event that $\mathcal{E}_{\infty}$ (and consequently $\Lambda$) is non-empty.

The process $\#\mathcal{E}_n$ is a Galton--Watson process (see for example \cite[Chapter 5.1]{zbMATH06653785}) $(Z_n)_{n\in \N}$ with $Z_0=1$ and offspring distribution $F=(p_k)_{k=0}^{M}$ 
\begin{equation*}
   p_k=\sum_{\substack{h\in H \\ \#h=k}}p_h. 
\end{equation*}
It follows from the theory of branching processes that the probability that the process survives (namely that $Z_n>0$ for all finite $n$) is $1-q$ where $q$ is the smallest non-negative fixed point of the probability generating function
\begin{equation*}
    f(s)=\sum_{k=0}^{M} p_k s^k.
\end{equation*}
From this it follows that the process survives (with positive probability) in the above sense if and only if $p_1=1$ (this is the critical case) or the expectation satisfies $\mathbb{E}(Z_1)>1$. 
In the event of survival, by Kőnig's lemma on the existence of infinite paths in locally finite trees, we have that $\mathcal{E}_{\infty}\neq \emptyset$; or $\Lambda\neq \emptyset$ since it is an intersection of a nested sequence of non-empty compact sets.

The substitution model (and the underlying Galton--Watson process) is called \emph{supercritical} if $\mathbb{E}(Z_1)>1$, \emph{critical} if $\mathbb{E}(Z_1)=1$, and \emph{subcritical} if $\mathbb{E}(Z_1)<1$.

\section{Branching property and statistical self-similarity}\label{a98}
For $\underline{i}\in \mathcal{E}_\ell, \underline{j}=j_1\dots j_n\in [M]_1^n$ let 
\begin{equation*}
    Q^{\underline{i}}_{\underline{j}}=W^{[\underline{i}]}_{j_1} W^{[\underline{i}j_1]}_{j_2}\dots W^{[\underline{i}j_1\dots j_{n-1}]}_{j_n}.
\end{equation*}
Consider the sub-process
\begin{equation*}
    \mathcal{E}_k^{(\underline{i})}=\{\underline{j}\in [M]_1^k: Q^{\underline {i}}_{\underline{j}}=1\}.
\end{equation*}
The process satisfies an important branching property: Let $\underline{i}_1,\dots, \underline{i}_m\in \mathcal{E}_\ell$, distinct level-$\ell$ nodes.
Then the sub-processes
\[
\bigl(\mathcal{E}_k^{(\underline{i}_1)}\bigr)_{k\ge 0},\dots,
\bigl(\mathcal{E}_k^{(\underline{i}_m)}\bigr)_{k\ge 0}
\]
are conditionally independent given $\sigma(\mathcal{E}_\ell)$, and each has the same distribution as the original process $(\mathcal{E}_k)_{k\ge 0}$.

At the level of sets, this property is called \textit{statistical self-similarity}.
\section{Inherited properties and a 0–1 law for trees}\label{a43}

In this section, we explain why certain events which occur with positive probability in fact occur with probability one, conditional on non-extinction.
This is a type of 0-1 law for Galton–Watson trees.
Note that this is not a consequence of the usual Kolmogorov 0-1 law as we do not have a nested sequence of independent $\sigma$-algebras to form a tail $\sigma$-algebra to which we could apply Kolmogorov's theorem.
Instead, we exploit the branching property and statistical self-similarity (see Section~\ref{a98}).

This idea has appeared in many books and papers on the topic of trees and independent percolation processes. We use the terminology of \cite[Proposition 5.6]{zbMATH06653785}.
We provide two interchangeable statement, one for labelled trees and one for compact sets. 
\begin{definition}
    A property $\mathcal{P}$ of a tree (resp.\ a compact set) is called an \emph{inherited property} if:
    \begin{itemize}
        \item The finite tree (resp.\ the empty set) has $\mathcal{P}$.
        \item If a tree $\mathcal{T}$ has $\mathcal{P}$, then for all vertices $\underline{i}$, the subtree $\mathcal{T}_{\underline{i}}$ has $\mathcal{P}$ (resp.\ If a compact set has $\mathcal{P}$ then each compact subset has $\mathcal{P}$ as well).
    \end{itemize}
\end{definition}

The key observation concerning inherited properties is the following lemma. 
\begin{lemma}\label{inherited}
    Suppose $\mathcal{P}$ is an inherited property of a tree (resp.\ of a compact set).
    Then the Galton–Watson tree (resp.\ the substitution random set constructed from the tree) satisfies $\mathcal{P}$ with probability either $0$ or $1$ conditional on non-extinction.
\end{lemma}
For the proof, see \cite[Proposition 5.6]{zbMATH06653785}.
\section{Large deviation lemma}
\begin{lemma}\label{large_dev_lem}
    Let $X_1, \dots, X_n$ be a sequence of finite i.i.d.\ random variables on $\R$ with $\mathbb{E}(X_1)=m$.
    Then for $a, b \in \R$, $a<m<b$ there exist $\gamma(a)\in(0, 1)$ and $\gamma(b)\in (0, 1)$ not depending on $n$ such that
    \begin{equation*}
        \mathbb{P}(X_1+\dots+X_n<n \cdot a)\leq \gamma(a)^n,\text{ and } \mathbb{P}(X_1+\dots+X_n>n \cdot b)\leq \gamma(b)^n.
    \end{equation*}
\end{lemma}
The proof is from \cite{toth2011statisztikus}, which we present in \cite[Section 1.10.3]{orgovanyi_thesis}.
\section{Conjugation of IFSs}\label{conjugation}
An IFS $\mathcal{S}'=\{S_i'\}$ is called \emph{bi-Lipschitz conjugate} to another IFS $\mathcal{S}=\{S_i\}$ if  $\#\mathcal{S}=\#\mathcal{S}'$ and there exists a permutation $\tau\colon\{1, \ldots, m\} \rightarrow\{1, \ldots, M\}$ and a bi-Lipschitz homeomorphism $h\colon\mathbb{R}^d \rightarrow \mathbb{R}^d$ such that ${S}_{\tau(i)}=h \circ S_i \circ h^{-1}$.
In this case the attractors satisfy $h(\Lambda_{\mathcal{S}})=\Lambda_{\mathcal{S}'}$ (see \cite[Section 1.3.3]{zbMATH07759243}).
From this it follows bi-Lipschitz conjugation of systems preserves the various properties studied in this note such as the dimension of sets, positivity of a measure of the set, and the existence of interior points.
We will use conjugation in the special case that $h$ is a similarity map, in which case the conjugated map has the same contraction ratio as the original system.  
\section*{Acknowledgements}
The authors thank Alex Rutar for his support. Further, as the manuscript grew out of parts of VO's PhD thesis, the authors thank Kenneth Falconer and Ville Suomala for their useful remarks on the thesis. We are grateful to Mark Pollicott and Polina Vytnova for their numerical estimates of the Lyapunov exponent.

Both authors received funding from the Hungarian Research Network through the HUN-REN--BME Stochastics Research Group. Furthermore, this work was supported by
the NKFI grants K142169, FK134251 and KKP144059
``Fractal geometry and applications'' Research Group. VO acknowledges support from the grant EKÖP-25-3-BME-2.

\section*{Tool and computational resource disclosure}
The AI model Claude (Anthropic) [Opus 5] was used mainly to improve the readability and presentation of the manuscript. It was also used to check parts of the proofs, which led to a few corrections and simplifications. The mathematical content and all arguments are the authors' own. The authors  take full responsibility for the correctness of the arguments and results and for the accuracy and completeness of the citations.
\bibliographystyle{plainnat} 
\bibliography{references}  

\end{document}